\documentclass[11pt]{amsart}

\usepackage[T1]{fontenc}
\usepackage[latin1]{inputenc}
\usepackage{a4wide}
\usepackage{setspace}
\usepackage{amsthm}
\usepackage[english]{babel}
\usepackage{latexsym}

\newtheorem{thm}{Theorem}[section]
\newtheorem{lem}[thm]{Lemma}

\newtheorem{defn}[thm]{Definition}
\newtheorem*{thm*}{Theorem}
\newtheorem{rem}[thm]{Remark}
\newtheorem{conj}[thm]{Conjecture}

\newtheorem{prop}[thm]{Proposition}

\newtheorem{cor}[thm]{Corollary}

\newenvironment{f-proof}[1][\sc D\'emonstration.]{\begin{trivlist}
\item[\hskip \labelsep {\bfseries #1}]}{\hfill{$\square$}\end{trivlist}}

\newcommand{\appl}[4]{
 \begin{array}{cccc}
  #1 & \longrightarrow & #2\\
  #3 & \longmapsto & #4
 \end{array}
}

\usepackage{amssymb}
\usepackage{amscd}
\usepackage{amsmath}
\usepackage{amsrefs}
\usepackage{stmaryrd}
\usepackage{xypic}

\newcommand{\bq}{\mathbb Q}

\newcommand{\bz}{\mathbb Z}
\newcommand{\br}{\mathbb R}
\newcommand{\bc}{\mathbb C}

\newcommand{\bF}{\mathbb F}

\newcommand{\bs}{\mathbb S}

\newcommand{\co}{\mathcal O}
\newcommand{\X}{\mathcal X}
\newcommand{\F}{\mathcal F}

\DeclareMathOperator{\Cone}{Cone}
\DeclareMathOperator{\Hom}{Hom} 

\DeclareMathOperator{\Ext}{Ext} 
\DeclareMathOperator{\Spec}{Spec}
\DeclareMathOperator{\fil}{Fil}
\DeclareMathOperator{\cone}{Cone}
\DeclareMathOperator{\rank}{rank}
\DeclareMathOperator{\id}{id}
\DeclareMathOperator{\trace}{Tr}
\DeclareMathOperator{\Tot}{Tot}

\newcommand{\mydet}{\mathrm{det}}

\theoremstyle{remark}

\DeclareMathOperator{\s}{Spec}
\newcommand{\D}{\mathcal D}

\begin{document}
\title[Special value conjectures and the functional equation]{Compatibility of special value conjectures with the functional equation of Zeta functions}
\author{Matthias Flach}
\author{Baptiste Morin}\thanks{The second author was supported by ANR-15-CE40-0002. The first author was supported by collaboration grant 522885 from the Simons foundation.}
\maketitle

\begin{abstract} We prove that the special value conjecture for the Zeta function $\zeta(\X,s)$ of a proper, regular arithmetic scheme $\X$ that we formulated in \cite{Flach-Morin-16}[Conj. 5.12] is compatible with the functional equation of $\zeta(\X,s)$ provided that the factor $C(\X,n)$ we were not able to compute in loc. cit. has the simple explicit form suggested in \cite{Flach-Morin-18}.
\end{abstract}

\section{Introduction}

This article is a continuation of our previous article \cite{Flach-Morin-16} in which we formulated a conjecture describing the leading Taylor coefficient of the Zeta function $\zeta(\X,s)$ of a proper regular arithmetic scheme $\X$ at integer arguments $n\in\bz$. Our conjecture involved a rather inexplicit correction factor $C(\X,n)\in\bq^\times$ which we could only compute for $\X=\Spec(\co_F)$ where $F$ is a number field all of whose completions $F_v$ are absolutely abelian. Based on this example a general formula for $C(\X,n)$ in terms of factorials 
\begin{equation}\label{explicitcorrectingfactor}
C(\X,n)^{-1}=\prod_{ i\leq n-1;\, j}(n-1-i)!^{(-1)^{i+j}\mathrm{dim}_{\bq}H^j(\X_{\bq},\Omega^i)}
\end{equation}
was suggested in \cite{Flach-Morin-18} and proven for $n=1$ in \cite{Flach-Siebel}. For $n\leq 0$ one has $C(\X,n)=1$ by definition. In this article we prove that our special value conjecture is compatible with the
functional equation of the Zeta function if $C(\X,n)$ is given by (\ref{explicitcorrectingfactor}), see Thm. \ref{thm:fe} below in this introduction. We regard this as convincing evidence that (\ref{explicitcorrectingfactor}) is indeed the right factor, even though we cannot yet prove that (\ref{explicitcorrectingfactor}) equals the original definition of $C(\X,n)$ in terms of $p$-adic Hodge theory. The original definition was made in such a way that our conjecture is compatible with the Tamagawa number conjecture of Bloch, Kato, Fontaine and Perrin-Riou \cite{bk88}, \cite{fpr91}. By restating our conjecture with the explicit factor (\ref{explicitcorrectingfactor}) we are in effect making a special value conjecture that is independent of $p$-adic Hodge theory and that is compatible with the functional equation of $\zeta(\X,s)$. Note that compatibility with the functional equation of motivic L-functions is not in general known for the Tamagawa number conjecture. Even for Tate motives 
 over number fields $F$ it is only known if all $F_v$ are absolutely abelian.

A precise statement of our special value conjecture is given in Conjecture \ref{conjmain} in subsection \ref{cyclic} of this introduction but not needed for this article. Neither do any of the results in sections \ref{verdier}-\ref{archimedean} of this article depend on unproven conjectures. Our main result Thm. \ref{thm:fe} follows from an unconditional theorem, Thm. \ref{thmintro}, which is stated in subsection \ref{statement} of this introduction.

\subsection{Cyclic homology and $C(\X,n)$} \label{cyclic}
We make a few remarks which do not pertain to the content of this paper but put formula (\ref{explicitcorrectingfactor}) in perspective. One might wonder about a more conceptual origin of the numerical factor $C(\X,n)$ and in order to say something about this question we first recall our special value conjecture in more detail. Let $\X$ be a regular scheme of dimension $d$, proper over $\Spec(\bz)$. Associated to $\X$ and $n\in\bz$ is an invertible $\bz$-module ("fundamental line")
\[ \Delta(\X/\bz,n):={\det}_\bz R\Gamma_{W,c}(\mathcal{X},\mathbb{Z}(n))\otimes_\bz {\det}_\bz R\Gamma(\X_{Zar},L\Omega_{\X/\bz}^{<n})\]
where $L\Omega_{\X/\bz}^{<n}$ is the derived de Rham complex \cite{Illusie72} modulo the $n$-th step in the Hodge filtration and $R\Gamma_{W,c}(\mathcal{X},\mathbb{Z}(n))$ is a perfect complex of abelian groups whose definition is dependent on assumptions  (finite generation of \'etale motivic cohomology, Artin-Verdier duality for torsion motivic cohomology) denoted by $\mathbf{L}(\overline{\mathcal{X}}_{et},n)$, $\mathbf{L}(\overline{\mathcal{X}}_{et},d-n)$, $\mathbf{AV}(\overline{\mathcal{X}}_{et},n)$ in \cite{Flach-Morin-16}. Also assuming the Beilinson conjectures in the form of conjecture  $\mathbf{B}(\mathcal{X},n)$ of \cite{Flach-Morin-16} one can construct a natural trivialization
\begin{equation}\lambda_\infty:\br\xrightarrow{\sim}\Delta(\X/\bz,n)\otimes_\bz\br.\label{introtriv}\end{equation}
Let $\zeta(\X,s)$ be the Zeta function of $\X$ and $\zeta^*(\X,n)\in\br^\times$ its leading Taylor coefficient at $s=n$. 

\begin{conj} \label{conjmain}
$$\lambda_{\infty}(\zeta^*(\mathcal{X},n)^{-1}\cdot C(\mathcal{X},n)\cdot\mathbb{Z})= \Delta(\mathcal{X}/\mathbb{Z},n)
$$
\end{conj}

This conjecture determines the real number $\zeta^*(\mathcal{X},n)\in\br$ up to sign. As noted above, the factor $C(\X,n)$ was originally defined in \cite{Flach-Morin-16}[Conj. 5.12] as the product over its $p$-primary parts where the definition of each $p$-part \cite{Flach-Morin-16}[Def. 5.6] involves $p$-adic Hodge theory and yet another assumption $\mathbf{D}_p(\mathcal{X},n)$ as well as assumption $\mathbf{R}(\mathbb{F}_p,\mathrm{dim}(\mathcal{X}_{\mathbb{F}_p}))$ borrowed from \cite{Geisser06}. We now prefer to simply restate Conjecture \ref{conjmain} with $C(\X,n)$ given by
(\ref{explicitcorrectingfactor}). Theorem \ref{thm:fe} will then show that Conjecture \ref{conjmain} holds for $(\X,n)$ if and only if it holds for $(\X,d-n)$, provided that $\zeta(\X,s)$ satisfies the expected functional equation (see Conjecture \ref{conj-funct-equation} in subsection \ref{statement}). 

It was shown in \cite{Flach-Morin-18}[Remark5.2]  that one can define a fairly natural modification $\tilde{L}\Omega_{\X/\bz}^{<n}$ of the derived deRham complex such that 
$$ \mydet_\bz R\Gamma(\X_{Zar},L\Omega_{\X/\bz}^{<n})=C(\X,n)\cdot \mydet_\bz R\Gamma(\X_{Zar},\tilde{L}\Omega_{\X/\bz}^{<n})$$
inside $\mydet_\bq R\Gamma(\X_{Zar},L\Omega_{\X/\bz}^{<n})_\bq\cong \mydet_\bq R\Gamma(\X_{\bq},\Omega_{\X_\bq/\bq}^{<n})$, leading to a version of Conjecture \ref{conjmain} without any correction factor. We would like to point out another modification of derived deRham cohomology that is perhaps even more natural than the definition of $\tilde{L}\Omega_{\X/\bz}^{<n}$ and should also explain formula (\ref{explicitcorrectingfactor}). Recall from \cite{antieau19} that there is a motivic filtration on cyclic homology $\fil_{Mot}^*HC(\X)$ with graded pieces given by derived deRham cohomology modulo the $n$-th step in the Hodge filtration
$$ gr_{Mot}^n HC(\X)\cong  R\Gamma(\X_{Zar},L\Omega_{\X/\bz}^{<n})[2n-2].$$
The corresponding spectral sequence already appears in \cite{majadas96}. Cyclic homology arises as $S^1$-homotopy-coinvariants on Hochschild homology $HC(\X)\cong HH(\X)_{S^1}$. One can consider the topological analogue and define
$$ TC^+(\X):= THH(\X)_{S^1} $$
where $THH$ denotes topological Hochschild homology (see for example \cite{Nikolaus-Scholze} for a review). Note that $TC^+(\X)$ is {\em not} what is usually called topological cyclic homology. We conjecture that there is a motivic filtration $\fil_{Mot}^*TC^+(\X)$ that maps to $\fil_{Mot}^*HC(\X)$ inducing an isomorphism
$$\fil_{Mot}^*TC^+(\X)_\bq\cong \fil_{Mot}^*HC(\X)_\bq.$$
Defining 
$$ R\Gamma(\X_{Zar},L\Omega_{\X/\bs}^{<n}):= gr_{Mot}^n TC^+(\X)[-2n+2]$$
we expect 
$$ \mydet_\bz R\Gamma(\X_{Zar},L\Omega_{\X/\bz}^{<n})=C(\X,n)\cdot \mydet_\bz R\Gamma(\X_{Zar},L\Omega_{\X/\bs}^{<n})$$
and therefore a version of Conjecture \ref{conjmain}
$$\lambda_{\infty}(\zeta^*(\mathcal{X},n)^{-1}\cdot\mathbb{Z})=\mathrm{det}_{\mathbb{Z}}R\Gamma_{W,c}(\mathcal{X},\mathbb{Z}(n))
\otimes_{\mathbb{Z}}R\Gamma(\X_{Zar},L\Omega_{\X/\bs}^{<n})$$
without correction factor. Here we may define $\mydet_\bz R\Gamma(\X_{Zar},L\Omega_{\X/\bs}^{<n})$ as the alternating determinant of the homotopy groups although we in fact expect  $R\Gamma(\X_{Zar},L\Omega_{\X/\bs}^{<n})$ to be a $H\bz$-module spectrum.

We can currently only define the motivic filtration if $\X$ is smooth proper over $\bF_p$, or for $\X=\Spec(\co_F)$. In the first case the motivic filtration was defined in \cite{bms19}, one can verify that both $R\Gamma(\X_{Zar},L\Omega_{\X/\bz}^{<n})$ and $R\Gamma(\X_{Zar},L\Omega_{\X/\bs}^{<n})$ have finite multiplicative Euler characteristic given by Milne's correction factor \cite{Flach-Morin-16}[Def. 5.4], and indeed one has $C(\X,n)=1$ by formula (\ref{explicitcorrectingfactor}). For $\X=\Spec(\co_F)$ the motivic filtration on $HC(\co_F)$ is
given by 
$$ \fil_{Mot}^n HC(\co_F) =\tau_{\geq 2n-3}HC(\co_F)$$
and one may define 
$$ \fil_{Mot}^n TC^+(\co_F) :=\tau_{\geq 2n-3}TC^+(\co_F)$$
by the same formula. Denote by $\mathcal D_F$  the different ideal of $\co_F$ and by $|D_F|=N\mathcal D_F$ the absolute value of the discriminant. As was shown in \cite{Flach-Morin-16}[1.6] there is an exact sequence
$$ 0\to HC_{2n-2}(\co_F)\to \co_F\to \Omega_{\co_F/\bz}(n) \to HC_{2n-3}(\co_F)\to 0$$
where $\Omega_{\co_F/\bz}(n)$ is a finite abelian group of cardinality $|D_F|^{n-1}$, i.e. we have
$$ |HC_{2n-3}(\co_F)|\cdot [\co_F:HC_{2n-2}(\co_F)]=|D_F|^{n-1}.$$
By a theorem of Lindenstrauss and Madsen \cite{ml2000} one has
$$THH_i(\co_F)=\begin{cases} \co_F &  i=0 \\ {\mathcal D}_F^{-1}/ j\cdot \co_F & i=2j-1\\ 0&\text{else.} \end{cases}$$
An easy analysis of the spectral sequence
$$ H_i(BS^1,THH_j(\co_F))\Rightarrow TC^+_{i+j}(\co_F) $$
then shows that $TC^+_{2n-3}(\co_F)$ is finite and $TC^+_{2n-2}(\co_F)\subseteq \co_F$ is a sublattice
so that  
\[ |TC^+_{2n-3}(\co_F)|\cdot [\co_F:TC^+_{2n-2}(\co_F)] = (n-1)!^{[F:\bq]}\cdot |D_F|^{n-1}. \]
And indeed one has $C(\Spec(\co_F),n)=(n-1)!^{-[F:\bq]}$ by (\ref{explicitcorrectingfactor}).

\subsection{Statement of the main result}\label{statement}

Under our standard assumptions $\mathbf{L}(\overline{\mathcal{X}}_{et},n)$, $\mathbf{L}(\overline{\mathcal{X}}_{et},d-n)$ and $\mathbf{AV}(\overline{\mathcal{X}}_{et},n)$, we defined in \cite{Flach-Morin-16} an exact triangle of perfect complexes of abelian groups
\begin{equation}\label{Wtri}
R\Gamma_{W,c}(\mathcal{X},\mathbb{Z}(n))\rightarrow R\Gamma_{W}(\overline{\mathcal{X}},\mathbb{Z}(n))\rightarrow R\Gamma_{W}(\mathcal{X}_{\infty},\mathbb{Z}(n)).
\end{equation}
Here $\overline{\mathcal{X}}$ is an Artin-Verdier compactification, $\mathcal{X}_{\infty}$ is the quotient topological space $\X(\bc)/G_\br$ and $$R\Gamma_{W}(\mathcal{X}_{\infty},\mathbb{Z}(n)):=R\Gamma(\mathcal{X}_{\infty},i_{\infty}^*\mathbb{Z}(n))$$
where $i_{\infty}^*\mathbb{Z}(n)$ is a certain complex of sheaves on $\mathcal{X}_{\infty}$, which is unconditionally defined.

In \cite{Flach-Morin-16}[5.7] we defined the invertible $\bz$-module 
\begin{eqnarray*}
\Xi_{\infty}(\mathcal{X}/\mathbb{Z},n)&:=&\mathrm{det}_{\mathbb{Z}}R\Gamma_{W}(\mathcal{X}_{\infty},\mathbb{Z}(n))\otimes\mathrm{det}^{-1}_{\mathbb{Z}} R\Gamma(\X_{Zar},L\Omega_{\X/\bz}^{<n}) \\
& &\otimes  \mathrm{det}^{-1}_{\mathbb{Z}}R\Gamma_{W}(\mathcal{X}_{\infty},\mathbb{Z}(d-n))
\otimes\mathrm{det}_{\mathbb{Z}}R\Gamma(\X_{Zar},L\Omega_{\X/\bz}^{<d-n})
\end{eqnarray*}
and a canonical trivialization 
$$\xi_{\infty}:\mathbb{R}\stackrel{\sim}{\longrightarrow} \Xi_{\infty}(\mathcal{X}/\mathbb{Z},n)\otimes\mathbb{R}.$$
We denote by $$x_{\infty}(\mathcal{X},n)^2\in\mathbb{R}_{>0}$$
the strictly positive real number such that $$\xi_{\infty}(x_{\infty}(\mathcal{X},n)^{-2}\cdot \mathbb{Z})=\Xi_{\infty}(\mathcal{X}/\bz,n)$$
and prove the following unconditional 
\begin{thm}\label{thmintro}
Let $\X$ be a regular scheme of dimension $d$, proper and flat over $\Spec(\bz)$.  We have 
\begin{equation*}
x_{\infty}(\mathcal{X},n)^{2}=\pm A(\mathcal{X})^{n-d/2} \cdot \frac{\zeta^*(\mathcal{X}_\infty,n)}{\zeta^*(\mathcal{X}_\infty,d-n)}\cdot \frac{C(\mathcal{X},n)}{C(\mathcal{X},d-n)}
\end{equation*}
where $A(\X)$ is the Bloch conductor (see Definition \ref{Bloch}) and $\zeta(\mathcal{X}_\infty,s)$ is the archimedean Euler factor of $\X$ (see Section \ref{archimedean}).
\end{thm}
We now explain the significance of this result. Let $\zeta(\overline{\mathcal{X}},s):=\zeta(\mathcal{X},s)\cdot \zeta(\mathcal{X}_\infty,s)$
be the completed Zeta-function of $\X$.
\begin{conj}\label{conj-funct-equation-intro} (Functional Equation) Let $\X$ be a regular scheme of dimension $d$, proper and flat over $\Spec(\bz)$. Then $\zeta(\X,s)$ has a meromorphic continuation to all $s\in\bc$ and 
$$A(\mathcal{X})^{(d-s)/2}\cdot \zeta(\overline{\mathcal{X}},d-s)=\pm A(\mathcal{X})^{s/2} \cdot \zeta(\overline{\mathcal{X}},s).$$
\end{conj}
Assume that $\mathbf{L}(\overline{\mathcal{X}}_{et},n)$, $\mathbf{L}(\overline{\mathcal{X}}_{et},d-n)$, $\mathbf{AV}(\overline{\mathcal{X}}_{et},n)$, $\mathbf{B}(\mathcal{X},n)$ and $\mathbf{B}(\mathcal{X},d-n)$ hold, so that Conjecture \ref{conjmain} for $(\X,n)$ and $(\X,d-n)$ makes sense. By  \cite{Flach-Morin-16}[Prop. 5.29], the exact triangle (\ref{Wtri}) and Weil-\'etale duality \cite{Flach-Morin-16}[Thm. 3.22] induce  a canonical isomorphism
$$\Delta(\mathcal{X}/\mathbb{Z},n)\otimes \Xi_{\infty}(\mathcal{X}/\mathbb{Z},n)\stackrel{\sim}{\longrightarrow}\Delta(\mathcal{X}/\mathbb{Z},d-n)$$
compatible with $\xi_{\infty}$ and the trivializations (\ref{introtriv}) of $\Delta(\X/\bz,n)$ and $\Delta(\X/\bz,d-n)$. We obtain 
\begin{thm}\label{thm:fe}
Assume that $\zeta(\overline{\X},s)$ satisfies Conjecture \ref{conj-funct-equation-intro}.
Then Conjecture \ref{conjmain} for $(\X,n)$ is equivalent to Conjecture \ref{conjmain} for $(\X,d-n)$.
\end{thm}

\subsection{Outline of this article} In section \ref{verdier} we study Verdier duality on the locally compact space $\X_\infty:=\X(\bc)/G_\br$ and how it applies to the complexes of sheaves $i^*_\infty\bz(n)$ introduced in \cite{Flach-Morin-16}[Def. 3.23]. The key result in terms of relevance for the following sections is Prop. \ref{final-bettiprop} which provides the correct power of $2$ appearing in the functional equation.

 In section \ref{bloch} we review duality results for the exterior powers of the cotangent complex $L_{\X/\bz}$ due to T. Saito \cite{tsaito04} and deduce duality for derived de Rham cohomology of $\X$. It turns out that the Bloch conductor $A(\X)$ of $\X$ introduced in \cite{bloch87a} measures the failure of a perfect duality for these theories, see Thm. \ref{lichtenbaum} and Prop. \ref{drduality}. Corollary \ref{dRfactor} then provides the correct power of $A(\X)$ appearing in the functional equation. 
 
 In section \ref{archimedean} we recall the archimedean Euler factors for $\zeta(\X,s)$ and make some preliminary computations towards the main result. 

Finally, in section \ref{sec:main} we prove Thm. \ref{thmintro} and Thm. \ref{thm:fe}, also employing the results already established in \cite{Flach-Morin-16}[Cor. 5.31] towards compatibility with the functional equation. 

\bigskip

{\em Acknowledgements.} We would like to thank S. Lichtenbaum for many indirect contributions to this project. Our realization that his Conjecture 0.1 in \cite{li18} could be proven using the ideas of T. Saito in \cite{tsaito88} was at the origin of this article (but in fact such a proof had already been carried out by T. Saito himself in \cite{tsaito04}[Cor. 4.9], see Thm. \ref{lichtenbaum} below).  Lichtenbaum's recent preprint \cite{li20} has considerable overlap with our article in that he also formulates a conjecture on special values of $\zeta(\X,s)$ and proves compatibility with the functional equations. Despite differences in language, and the fact that all results of \cite{li20} are only up to powers of $2$, we believe our approaches are largely equivalent. 

We would also like to thank Spencer Bloch for interesting discussions related to $C(\X,n)$.

\section{Verdier duality on $\X_\infty=\X(\bc)/G_\br$}\label{verdier}

\subsection{Statement of the duality theorem}
Let $\X$ be a regular, flat and proper scheme over  $\s(\bz)$. Assume that $\X$ is connected of dimension $d$. We denote by $\X_{\infty}:=\X(\bc)/G_{\br}$ the quotient topological space, where $\X(\bc)$ is endowed with the complex topology. Let
$$p:\X(\bc)\rightarrow \X_{\infty}$$
be the quotient map and let
$$\pi:\mathrm{Sh}(G_{\br},\X(\bc))\rightarrow \mathrm{Sh}(\X_{\infty})$$
be the canonical morphism of topoi, where $\mathrm{Sh}(G_{\br},\X(\bc))$ is the category of $G_{\br}$-equivariant sheaves on $\X(\bc)$. We have the formula
$$\pi_*(\mathcal{F})\simeq (p_*\mathcal{F})^{G_{\br}}.$$
Let $\bz(n):=(2i\pi)^n\cdot\bz$ be the abelian sheaf on $\mathrm{Sh}(G_{\br},\X(\bc))$ defined by the obvious $G_{\br}$-action on $(2i\pi)^n\cdot\bz$. In \cite{Flach-Morin-16}[Def. 3.23], we defined the complex of sheaves on $\X_{\infty}$
$$i^*_{\infty}\bz(n):=\mathrm{Fib}(R\pi_*\bz(n)\rightarrow \tau^{>n}R\widehat{\pi}_*\bz(n))$$
for any $n\in\bz$. We define similarly
$$Ri^!_{\infty}\bz(n+1)[3]:=\bz'(n):=\mathrm{Fib}(R\pi_*\bz(n)\rightarrow \tau^{\geq n}R\widehat{\pi}_*\bz(n))$$
and we set 
$$e:=d-1.$$
If $Z$ is a locally compact topological space, we denote by $\mathcal{D}_Z:=Rf^!\bz$ the dualizing complex, where $f:Z\rightarrow\{*\}$ is the map from $Z$ to the point.

\begin{thm}\label{dualitythm}
There is an equivalence $\bz'(e)\stackrel{\sim}{\rightarrow} \D_{\X_{\infty}}[-2e]$ and a perfect pairing
$$i^*_{\infty}\bz (n) \otimes^L \bz'(e-n) \longrightarrow \bz'(e)\stackrel{\sim}{\longrightarrow} \D_{\X_{\infty}}[-2e]$$
in the derived category of abelian sheaves over $\X_{\infty}$, for any $n\in\bz$.
\end{thm}
\begin{proof}
We set $\bz(n):=i^*_{\infty}\bz(n)$, we denote by $\iota:\X(\br)\rightarrow\X_{\infty}$ the closed immersion and by $j$ the complementary open immersion.
By Proposition \ref{productmap} there is a product map
$$\bz(n)\otimes^L \bz'(e-n)\rightarrow \mathcal{D}_{\X_{\infty}}[-2e]$$
inducing
\begin{equation}\label{dualitymap0}
\bz(n)\rightarrow R\underline{\mathrm{Hom}}(\bz'(e-n),\mathcal{D}_{\X_{\infty}}[-2e]).
\end{equation}
Then (\ref{dualitymap0}) induces an equivalence
$$j^*\bz(n)\stackrel{\sim}{\rightarrow} j^* R\underline{\mathrm{Hom}}(\bz'(e-n),\mathcal{D}_{\X_{\infty}}[-2e])$$
by Proposition \ref{j}. Similarly, (\ref{dualitymap0}) induces an equivalence
$$R\iota^!\bz(n)\stackrel{\sim}{\rightarrow} R\iota^! R\underline{\mathrm{Hom}}(\bz'(e-n),\mathcal{D}_{\X_{\infty}}[-2e])$$
by Proposition \ref{iota}. It follows that (\ref{dualitymap0}) 
is an equivalence. Applying $R\underline{\mathrm{Hom}}(-,\mathcal{D}_{\X_{\infty}}[-2e])$, we get an equivalence
$$\bz'(e-n)\stackrel{\sim}{\rightarrow} R\underline{\mathrm{Hom}}(\bz(n),\mathcal{D}_{\X_{\infty}}[-2e]).$$
Since $\bz(0)$ is the constant sheaf $\bz$, we have 
$$\bz'(e)\stackrel{\sim}{\longrightarrow} \D_{\X_{\infty}}[-2e].$$
\end{proof}
We immediately obtain
\begin{cor}\label{dualitycor}
There is a trace map $R\Gamma(\X_{\infty},\bz'(e))\rightarrow \bz[-2e]$
and a perfect pairing
$$R\Gamma(\X_{\infty},i^*_{\infty}\bz(n))\otimes^L R\Gamma(\X_{\infty},\bz'(e-n))\rightarrow R\Gamma(\X_{\infty},\bz'(e))\rightarrow \bz[-2e]$$
of perfect complexes of abelian groups, for any $n\in\bz$.
\end{cor}

The following corollaries also follow easily from Theorem \ref{dualitythm}. We state them in order to justify the notation $Ri^!_{\infty}\bz(n)$.

\begin{cor}
There is a trace map
$$R\Gamma(\X_{\infty},Ri^!_{\infty}\bz(d))\rightarrow \bz[-2d-1]$$
and a perfect pairing
$$R\Gamma(\X_{\infty},i^*_{\infty}\bz(n))\otimes^L R\Gamma(\X_{\infty},Ri^!_{\infty}\bz(d-n))\rightarrow R\Gamma(\X_{\infty},Ri^!_{\infty}\bz(d))\rightarrow \bz[-2d-1]$$
of perfect complexes of abelian groups, for any $n\in\bz$.
\end{cor}
\begin{cor} Assume that $\X$ satisfies the assumptions  $\mathbf{L}(\overline{\mathcal{X}}_{et},n)$, $\mathbf{L}(\overline{\mathcal{X}}_{et},d-n)$ and $\mathbf{AV}(\overline{\mathcal{X}}_{et},n)$ of \cite{Flach-Morin-16}[3.2].  We define 
$$R\Gamma_W(\X,\bz(n)):=R\mathrm{Hom}(R\Gamma_{W,c}(\X,\bz(d-n)),\mathbb{Z}[-2d-1]).$$
Then we have an exact triangle
$$R\Gamma(\X_{\infty},Ri^!_{\infty}\bz(n))\rightarrow R\Gamma_W(\overline{\X},\bz(n))\rightarrow R\Gamma_W(\X,\bz(n)).$$
\end{cor}

\subsection{Proof of the duality theorem}
The proof of Theorem \ref{dualitythm} relies on the results proven below.

\subsubsection{Notations}

We denote by
$\iota:\X(\br)\rightarrow \X_{\infty}$
the closed immersion and  by $j:\X^{\circ}_{\infty}\rightarrow \X_{\infty}$ the complementary open immersion, where $\X^{\circ}_{\infty}:=\X_{\infty}-\X(\br)$. We set $\X(\bc)^{\circ}:=\X(\bc)-\X(\br)$. We denote by
$$p^{\circ}:\X(\bc)^{\circ}\rightarrow  \X^{\circ}_{\infty}$$
the quotient map, and by
$$\pi^{\circ}:\mathrm{Sh}(G_{\br},\X(\bc)^{\circ})\stackrel{\sim}{\rightarrow} \mathrm{Sh}( \X^{\circ}_{\infty})$$ 
the morphism of topoi induced by $\pi$, which is an equivalence since $G_{\br}$ has no fixed point on $\X(\bc)^{\circ}$. If $x\in\X(\br)$ then we denote $\iota_x:x\rightarrow \X(\br)$ (or $\iota_x:x\rightarrow \X_{\infty}$) the inclusion. The complex of sheaves over $\X_{\infty}$ denoted by $\bz(n)$ always refers to $i^*_{\infty}\bz(n)$.

We denote by $C^*(G_{\br},\bz(n)):=R\Gamma(G_{\br},\bz(n))$ the cohomology of $G_{\br}$ with coefficients in $(2i\pi)^n\bz$, by $\widehat{C}^*(G_{\br},\bz(n)):=R\widehat{\Gamma}(G_{\br},\bz(n))$ Tate cohomology, and by 
$C_*(G_{\br},\bz(n))$ the homology of $G_{\br}$ with coefficients in $(2i\pi)^n\bz$. We have a fiber sequence
$$C_*(G_{\br},\bz(n))\rightarrow C^*(G_{\br},\bz(n))\rightarrow \widehat{C}^*(G_{\br},\bz(n)).$$
Recall that, if $Z$ is a locally compact topological space, we denote by $\mathcal{D}_Z:=Rf^!\bz$ the dualizing complex, where $f:Z\rightarrow\{*\}$ is the map from $Z$ to the point.

\subsubsection{The duality map}

\begin{prop}\label{productmap}
For any $n\in\bz$, there is a canonical map
$$i^*_{\infty}\bz(n)\otimes^L \bz'(e-n)\rightarrow \mathcal{D}_{\X_{\infty}}[-2e]$$
in the derived category of abelian sheaves over $\X_{\infty}$.
\end{prop}
\begin{proof} 
Let $f$ be the map from $\X_{\infty}$ to the point. We start with the morphism
$$i^*_{\infty}\bz(n)\otimes^L \bz'(e-n)\rightarrow R\pi_*((2 i\pi)^n\bz) \otimes^L R\pi_*((2 i\pi)^{e-n}\bz)\rightarrow R\pi_*((2 i\pi)^e\bz).$$
Then the map
$$R\pi_*((2 i\pi)^e\bz)\rightarrow \mathcal{D}_{\X_{\infty}}[-2e]:=f^!\bz [-2e]$$
is given by
$$Rf_!R\pi_*((2 i\pi)^e\bz)\simeq R\Gamma(G_{\br},\X(\bc), (2 i\pi)^e\bz)\rightarrow \bz [-2e] $$
where the last map is 
$$R\Gamma(G_{\br},\X(\bc), (2 i\pi)^e\bz)\rightarrow R\Gamma(\X(\bc), (2 i\pi)^e\bz)\rightarrow\tau^{\geq 2e} R\Gamma(\X(\bc), (2 i\pi)^e\bz)\rightarrow \bz [-2e].$$
Note that $Rf_!=Rf_*$ since $\X_{\infty}$ is compact.
\end{proof}

\begin{defn}
For any $n\in\bz$, we consider the morphism
\begin{equation}\label{dualitymap}
\bz(n)\rightarrow R\underline{\mathrm{Hom}}(\bz'(e-n),\mathcal{D}_{\X_{\infty}}[-2e])
\end{equation}
induced by the product map above.
\end{defn}

\subsubsection{The map $j^*\bz(n)\rightarrow j^*R\underline{\mathrm{Hom}}(\bz'(e-n), \D_{\X_{\infty}}[-2e])$}
\begin{prop}\label{j}
The  canonical map 
$$j^*\bz(n)\rightarrow j^*R\underline{\mathrm{Hom}}(\bz'(e-n), \D_{\X_{\infty}}[-2e])$$
is an equivalence.
\end{prop}
\begin{proof}We replace $n$ by $e-n$. We have
\begin{eqnarray*}
j^*R\underline{\mathrm{Hom}}(\bz'(n), \D_{\X_{\infty}}[-2e])&\simeq& R\underline{\mathrm{Hom}}_{\mathrm{Sh}(\X^{\circ}_{\infty})}(j^*\bz'(n), \D_{\X^{\circ}_{\infty}}[-2e])\\
&\simeq& R\underline{\mathrm{Hom}}_{\mathrm{Sh}(\X^{\circ}_{\infty})}(\pi_*^{\circ}(2i\pi)^n\bz, \D_{\X^{\circ}_{\infty}})[-2e]. 
\end{eqnarray*}
Similarly, we have $j^*\bz(e-n)=\pi_*^{\circ}(2i\pi)^{e-n}\bz$. So we need to check that the map
$$\pi_*^{\circ}(2i\pi)^{e-n}\bz\rightarrow R\underline{\mathrm{Hom}}_{\mathrm{Sh}(\X^{\circ}_{\infty})}(\pi_*^{\circ}(2i\pi)^n\bz, \D_{\X^{\circ}_{\infty}})[-2e]$$
is an equivalence. The map $p^{\circ}: \X(\bc)^{\circ}\rightarrow \X^{\circ}_{\infty}$ is a finite \'etale Galois cover, hence $p^{\circ,*}$ is conservative.
Hence it is enough to check that
$$p^{\circ,*}\pi_*^{\circ}(2i\pi)^{e-n}\bz\rightarrow  R\underline{\mathrm{Hom}}_{\mathrm{Sh}(\X(\bc)^{\circ})}(p^{\circ,*}\pi_*^{\circ}(2i\pi)^n\bz, \D_{\X(\bc)^{\circ}})[-2e]$$
is an equivalence. But we have
$$p^{\circ,*}\pi_*^{\circ}(2i\pi)^{n}\bz \simeq (2i\pi)^{n}\bz,$$
hence one is reduced to observe that
$$(2i\pi)^{e-n}\bz\rightarrow  R\underline{\mathrm{Hom}}_{\mathrm{Sh}(\X(\bc)^{\circ})}((2i\pi)^n\bz, \D_{\X(\bc)^{\circ}})[-2e]$$
is an equivalence by Verdier duality on the complex (hence orientable) manifold $\X(\bc)^{\circ}$.
\end{proof}

\subsubsection{The complex $\iota_x^*R\iota^!\bz(n)$}

\begin{lem} For any $n\in\bz$ and any $x\in\X(\br)$, we have a fiber sequence
$$R\Gamma(G_\br, \bz(n))\rightarrow \iota_x^* Rj_*j^*\bz(n) \rightarrow R\Gamma(G_{\br},\bz(n-e))[-(e-1)]$$
and $\iota_x^* Rj_*j^*\bz(n)$ is cohomologically concentrated in  degrees $\in [0,e-1]$.
\end{lem}
\begin{proof}
For $e=0$, the map $j$ is both a closed and an open immersion hence $\iota_x^* Rj_*j^*\bz(n)=0$. So the result is obvious in that case, hence we may assume $e\geq 1$.

Note first that $j^*\bz(n)\simeq R\pi^{\circ}_*((2i\pi)^n\bz)$. Let $x\in\X(\br)\subset\X(\bc)$. For a point $z\in\X(\bc)$ in the neighbourhood of $x$, we have $$z=(a_1,b_1,\cdots,a_e,b_e) \in \bc^{e}= (\br\oplus i\cdot \br)^{e}$$ where $\sigma$ acts as follows
$$(a_1,\cdots ,a_e,b_1,\cdots,b_e)\mapsto (a_1,\cdots , a_e,-b_1,\cdots,-b_e)\in \br^{e}\oplus i\cdot \br ^{e}.$$
So a basic open neighborhood of $x\in\X(\br)$ in $\X(\bc)$ is of the form $B^{e}\times B^{e}$
where $B^{e}$ denotes an open ball in $\br^{e}$, and $\sigma$ acts trivially on the first ball and by multiplication by $-1$ on the second ball. We have $$\X(\br)\cap (B^{e}\times B^{e})= B^{e}\times 0$$ and a $G_{\br}$-equivariant homotopy equivalence
$$\X(\bc)^{\circ}\cap (B^{e}\times B^{e})= B^{e}\times (B^{e}-0)\simeq B^{e}\times {\bf S}^{e-1}\simeq {\bf S}^{e-1}$$
where  $G_{\br}$ acts by its antipodal action on the $(e-1)$-sphere  ${\bf S}^{e-1}$. We obtain
\begin{eqnarray*}
\iota_x^* Rj_*j^*\bz(n)&\simeq& \mathrm{colim}_{x\in U\subset\X_{\infty}} R\Gamma(U-\X(\br),\bz(n))\\
&\simeq&  \mathrm{colim}_{x\in U\subset\X_{\infty}} R\Gamma(G_{\br},p^{-1}(U-\X(\br)),\bz(n))\\
&\simeq& R\Gamma(G_{\br},{\bf S}^{e-1}, \bz(n)) 
\end{eqnarray*}
where $G_{\br}$ acts both on  ${\bf S}^{e-1}$ and $\bz(n):=(2i\pi)^n\bz$. But we have a fiber sequence in the derived category of $\bz[G_{\br}]$-modules
$$\bz(n)\rightarrow R\Gamma({\bf S}^{e-1},\bz(n))\rightarrow \bz(n-e)[-(e-1)]$$
where the boundary map $\bz(n-e))[-(e-1)]\rightarrow \bz(n)[1]$ is the non-trivial class in 
$$\mathrm{Hom}_{\bz[G_{\br}]}(\bz(n-e))[-(e-1)], \bz(n)[1])\simeq \mathrm{Hom}_{\bz[G_{\br}]}(\bz, \bz(e)[e])\simeq H^{e}(G_{\br},\bz(e))\simeq \bz/2\bz.$$
Indeed, it must be the non-trivial class because
 $$R\Gamma(G_{\br},{\bf S}^{e-1}, \bz(n))\simeq R\Gamma({\bf S}^{e-1}/\{\pm 1\}, \bz(n))$$
is cohomologically concentrated in  degrees $\in [0,e-1]$ since ${\bf S}^{e-1}/\{\pm 1\}$ is a $(e-1)$-manifold.
\end{proof}

\begin{lem} For any $n\in\bz$, we have
$$ \iota_x^*R\iota^!\bz(n)\simeq \mathrm{Fib}\left(R\Gamma(G_{\br},\bz(n-e))[-e]\rightarrow\tau^{> n}R\widehat{\Gamma}(G_\br, \bz(n))\right).$$
\end{lem}
\begin{proof}
First we assume $n\geq 0$, so that $\iota_x^* \bz(n)\simeq \tau^{\leq n}R\Gamma(G_\br, \bz(n))$. Then we have the following diagram with exact rows and columns:
\[ \xymatrix{
\tau^{> n}R\Gamma(G_\br, \bz(n))[-1]\ar[d]\ar[r]^{} &\tau^{\leq n}R\Gamma(G_\br, \bz(n))\ar[d]^{} \ar[r]^{}&R\Gamma(G_\br, \bz(n))\ar[d] \\
\iota_x^*R\iota^!\bz(n)\ar[r]^{}\ar[d]&\iota_x^* \bz(n) \ar[r]^{}\ar[d]& \iota_x^* Rj_*j^*\bz(n) \ar[d]\\
R\Gamma(G_{\br},\bz(n-e))[-e]\ar[r]^{}&0 \ar[r]^{}& R\Gamma(G_{\br},\bz(n-e))[-(e-1)] 
}
\]
Now we assume $n< 0$. By Lemma \ref{lem-simplify}, we have an equivalence
$$\iota_x^* \bz(n)\simeq \tau_{\leq - n-2}C_*(G_\br, \bz(n))$$
where both sides vanish for $n=-1$. We obtain the following diagram with exact rows and columns:
\[ \xymatrix{
(\tau^{> n}\widehat{C}(G_\br, \bz(n)))[-1]\ar[d]\ar[r]^{} &\tau_{\leq - n-2}C_*(G_\br, \bz(n))\ar[d]^{} \ar[r]^{}&C^{*}(G_\br, \bz(n))\ar[d] \\
\iota_x^*R\iota^!\bz(n)\ar[r]^{}\ar[d]&\iota_x^*\bz(n) \ar[r]^{}\ar[d]& \iota_x^*Rj_*j^*\bz(n) \ar[d]\\
C^*(G_{\br},\bz(n-e))[-e]\ar[r]^{}&0 \ar[r]^{}& C^*(G_{\br},\bz(n-e))[-(e-1)] 
}
\]
\end{proof}

\begin{prop}\label{iota1}
For $n < e$, we have
$$ \iota_x^*R\iota^!\bz(n)\simeq (\tau_{\leq e-n-2} C_*(G_\br, \bz(n-e)))[-e].$$
For $ n\geq e$, we have
$$ \iota_x^*R\iota^!\bz(n)\simeq (\tau^{\leq n-e} R\Gamma(G_\br, \bz(n-e)))[-e].$$
\end{prop}
\begin{proof}
We have
$$\tau^{> n}R\widehat{\Gamma}(G_\br, \bz(n))\simeq (\tau^{> n-e}R\widehat{\Gamma}(G_\br, \bz(n-e)))[-e]$$
and an equivalence
$$ \iota_x^*R\iota^!\bz(n)\simeq \mathrm{Fib}\left(R\Gamma(G_{\br},\bz(n-e))\rightarrow\tau^{> n-e}R\widehat{\Gamma}(G_\br, \bz(n-e))\right) [-e].$$
Hence the result follows from Lemma \ref{lem-simplify} below.
\end{proof}

\begin{lem}\label{lem-simplify}
For any $m\geq 0$, we have an equivalence
$$\tau^{\leq m}R\Gamma(G_{\br},\bz(m))\simeq \mathrm{Fib} \left(R\Gamma(G_{\br},\bz(m))\rightarrow \tau^{>m}R\widehat{\Gamma}(G_\br,\bz(m))\right).$$
Similarly, for any $m< 0$, we have
$$\tau_{\leq -m-2}C_*(G_{\br},\bz(m))\simeq \mathrm{Fib} \left(R\Gamma(G_{\br},\bz(m))\rightarrow \tau^{>m}R\widehat{\Gamma}(G_\br,\bz(m))\right).$$
\end{lem}

\begin{proof}
The first assertion is obvious. The second equivalence holds for $m=-1$ since both side vanish. It remains to show that the second equivalence holds for $m\leq-2$. We have the following exact diagram
\[ \xymatrix{
(\tau^{\leq m}\widehat{C}(G_\br, \bz(m)))[-1]\ar[d]\ar[r]^{} &0\ar[d]^{} \ar[r]^{}&\tau^{\leq m}\widehat{C}^*(G_{\br},\bz(m))\ar[d] \\
C_*(G_{\br},\bz(m))\ar[r]^{}\ar[d]&C^*(G_{\br},\bz(m)) \ar[r]^{}\ar[d]& \widehat{C}^*(G_{\br},\bz(m)) \ar[d]\\
F\ar[r]^{}&C^*(G_{\br},\bz(m)) \ar[r]^{}& \tau^{>m}\widehat{C}^*(G_{\br},\bz(m)) 
}
\]
hence a cofiber sequence
$$(\tau^{\leq m}\widehat{C}(G_\br, \bz(m)))[-1]\rightarrow  C_*(G_{\br},\bz(m)) \rightarrow F.$$
In view of the equivalences
$$(\tau^{\leq m}\widehat{C}(G_\br, \bz(m)))[-1]\simeq \tau^{\leq m+1}(\widehat{C}(G_\br, \bz(m))[-1])\simeq \tau^{\leq m+1}C_*(G_\br, \bz(m))$$
we obtain
$$F\simeq \tau^{> m+1}C_*(G_\br, \bz(m))= \tau^{\geq m+2}C_*(G_\br, \bz(m))= \tau_{\leq -m-2}C_*(G_\br, \bz(m)).$$
\end{proof}

\begin{lem}\label{lem-simplify2}
For any $m> 0$, we have an equivalence
$$\tau^{< m}R\Gamma(G_{\br},\bz(m))\simeq \mathrm{Fib} \left(R\Gamma(G_{\br},\bz(m))\rightarrow \tau^{\geq m }R\widehat{\Gamma}(G_\br,\bz(m))\right).$$
Similarly, for any $m\leq 0$, we have
$$\tau_{\leq -m}C_*(G_{\br},\bz(m))\simeq \mathrm{Fib} \left(R\Gamma(G_{\br},\bz(m))\rightarrow \tau^{\geq m}R\widehat{\Gamma}(G_\br,\bz(m))\right).$$
\end{lem}
\begin{proof}
The first assertion is obvious. The second equivalence for $m=0$ follows from the exact sequence
$$0=\widehat{H}^{-1}(G_{\br},\bz)\rightarrow H_{0}(G_{\br},\bz)\rightarrow H^{0}(G_{\br},\bz)\rightarrow \widehat{H}^{0}(G_{\br},\bz)\rightarrow 0$$
and the isomorphism $H^{i}(G_{\br},\bz)\stackrel{\sim}{\rightarrow} \widehat{H}^{i}(G_{\br},\bz)$ for $i>0$.

It remains to show that the second equivalence holds for $m\leq-1$. We have the following exact diagram
\[ \xymatrix{
(\tau^{< m}\widehat{C}(G_\br, \bz(m)))[-1]\ar[d]\ar[r]^{} &0\ar[d]^{} \ar[r]^{}&\tau^{< m}\widehat{C}^*(G_{\br},\bz(m))\ar[d] \\
C_*(G_{\br},\bz(m))\ar[r]^{}\ar[d]&C^*(G_{\br},\bz(m)) \ar[r]^{}\ar[d]& \widehat{C}^*(G_{\br},\bz(m)) \ar[d]\\
F\ar[r]^{}&C^*(G_{\br},\bz(m)) \ar[r]^{}& \tau^{\geq m}\widehat{C}^*(G_{\br},\bz(m)) 
}
\]
hence a cofiber sequence
$$(\tau^{< m}\widehat{C}(G_\br, \bz(m)))[-1]\rightarrow  C_*(G_{\br},\bz(m)) \rightarrow F.$$
In view of the equivalences
$$(\tau^{< m}\widehat{C}(G_\br, \bz(m)))[-1]\simeq \tau^{< m+1}(\widehat{C}(G_\br, \bz(m))[-1])\simeq \tau^{< m+1}C_*(G_\br, \bz(m))$$
we obtain
$$F\simeq \tau^{\geq m+1}C_*(G_\br, \bz(m))= \tau_{\leq -m-1}C_*(G_\br, \bz(m))\simeq \tau_{\leq -m}C_*(G_\br, \bz(m))$$
since
$$H_{-m}(G_\br, \bz(m))=\widehat{H}^{m-1}(G_\br, \bz(m))=0.$$
\end{proof}

\subsubsection{The complex $R\iota^!R\underline{\mathrm{Hom}}(\bz'(e-n), \D_{\X_{\infty}}[-2e])$}
We  denote by $f:\X(\br)\rightarrow \{*\}$ the map from $\X(\br)$ to the point and
we denote by $\omega_{\X(\br)}$ the orientation sheaf on the $e$-manifold $\X(\br)$. We have 
$$\mathcal{D}_{\X(\br)}:=f^!\bz\simeq \omega_{\X(\br)}[e].$$

\begin{prop}\label{iota2}
For $e-n>0$ we have
$$R\iota^!R\underline{\mathrm{Hom}}(\bz'(e-n), \D_{\X_{\infty}}[-2e])\simeq f^*(\tau_{\leq e-n-2}C_*(G_{\br},\bz(e-n)))\otimes^L\omega_{\X(\br)}[-e].$$
\end{prop}
\begin{proof}
Using Lemma \ref{lem-simplify2}  and Lemma \ref{dualitybasic}, we obtain
\begin{eqnarray*}
R\iota^!R\underline{\mathrm{Hom}}(\bz'(e-n), \D_{\X_{\infty}})[-2e]&\simeq& R\underline{\mathrm{Hom}}(\iota^*\bz'(e-n), \D_{\X(\br)})[-2e]\\
&\simeq& R\underline{\mathrm{Hom}}(f^*\tau^{<e-n}R\Gamma(G_{\br},\bz(e-n)), \D_{\X(\br)})[-2e]\\
&\simeq& f^!R\mathrm{Hom}(\tau^{<e-n}R\Gamma(G_{\br},\bz(e-n)), \bz)[-2e]\\
&\simeq& f^!R\underline{\mathrm{Hom}}(\tau^{\leq e-n-2}R\Gamma(G_{\br},\bz(e-n)), \bz)[-2e]\\
&\simeq& f^!(\tau_{\leq e-n -2}C_*(G_{\br},\bz(e-n)))[-2e]\\
&\simeq& f^*(\tau_{\leq e-n-2}C_*(G_{\br},\bz(e-n)))\otimes^L\omega_{\X(\br)}[-e].
\end{eqnarray*}
\end{proof}

\begin{prop}\label{iota3}
For  $e-n\leq 0$ we have
$$R\iota^!R\underline{\mathrm{Hom}}(\bz'(e-n),\D_{\X_{\infty}}[-2e])\simeq f^*(\tau^{\leq n-e}R\Gamma(G_{\br},\bz(n-e)))\otimes^L\omega_{\X(\br)}[-e].$$
\end{prop}
\begin{proof}
Using Lemma \ref{lem-simplify2} and Lemma \ref{dualitybasic}, we obtain
\begin{eqnarray*}
R\iota^!R\underline{\mathrm{Hom}}(\bz'(e-n), \D_{\X_{\infty}})[-2e]&\simeq& R\underline{\mathrm{Hom}}(\iota^*\bz'(e-n), \D_{\X(\br)})[-2e]\\
&\simeq& R\underline{\mathrm{Hom}}(f^*\tau_{\leq n-e}C_*(G_{\br},\bz(e-n)), \D_{\X(\br)})[-2e]\\
&\simeq& f^!R\mathrm{Hom}(\tau_{\leq n-e}C_*(G_{\br},\bz(e-n)), \bz)[-2e]\\
&\simeq & f^!(\tau^{\leq n-e}R\Gamma(G_{\br},\bz(n-e)))[-2e]\\
&\simeq& f^* (\tau^{\leq n-e}R\Gamma(G_{\br},\bz(n-e)))\otimes^L\omega_{\X(\br)}[-e].
\end{eqnarray*}
\end{proof}

\begin{lem}\label{dualitybasic}
For any $n\in \bz$, the pairing
$$C_*(G_{\br},\bz(-n))\otimes_{\bz}^L C^*(G_{\br},\bz(n)) \rightarrow C_*(G_{\br},\bz(0))\rightarrow \bz[0]$$
induces a perfect pairing
$$\tau_{\leq n}C_*(G_{\br},\bz(-n))\otimes_{\bz}^L \tau^{\leq n}C^*(G_{\br},\bz(n)) \rightarrow \bz[0]$$
of perfect complexes of abelian groups.
\end{lem}
\begin{proof} The result is trivial for $n<0$ and clear for $n=0$. So we assume $n>0$.
The pairing induces an equivalence
$$C^*(G_{\br},\bz(n))\rightarrow R\mathrm{Hom}(C_*(G_{\br},\bz(-n)),\bz)$$
hence it is enough to observe that
$$\tau^{\leq n} R\mathrm{Hom}(C_*(G_{\br},\bz(-n)),\bz)\simeq R\mathrm{Hom}(\tau_{\leq n}C_*(G_{\br},\bz(-n)),\bz).$$
For any cohomological complex $A^*$, we have a short exact sequence
$$0\rightarrow \mathrm{Ext}(H^{-i+1}(A^*),\bz)\rightarrow H^{i}(R\mathrm{Hom}(A^*,\bz))\rightarrow \mathrm{Hom}(H^{-i}(A^*),\bz)\rightarrow 0.$$
We obtain 
$$H^{i}(R\mathrm{Hom}(\tau_{\leq n}C_*(G_{\br},\bz(-n)),\bz))=H^{i}(R\mathrm{Hom}(C_*(G_{\br},\bz(-n)),\bz))$$
for $i\leq n$ and $i> n+1$. Since we have $H_n(G_{\br},\bz(-n))=0$ for any $n>0$, we get 
$$H^{n+1}(R\mathrm{Hom}(\tau_{\leq n}C_*(G_{\br},\bz(-n)),\bz))=0.$$

\end{proof}

\begin{rem}
For $n>0$, we have $H_n(G_{\br},\bz(-n))=0$ hence
 $$\tau_{\leq n}C_*(G_{\br},\bz(-n))\simeq \tau_{< n}C_*(G_{\br},\bz(-n)).$$
\end{rem}

\subsubsection{The map $R\iota^!\bz(n)\rightarrow R\iota^!R\underline{\mathrm{Hom}}(\bz'(e-n), \D_{\X_{\infty}}[-2e])$}

\begin{prop} \label{iota}
The map $$R\iota^!\bz(n)\rightarrow R\iota^!R\underline{\mathrm{Hom}}(\bz'(e-n), \D_{\X_{\infty}}[-2e])$$
is an equivalence.
\end{prop}
\begin{proof}
For $e-n>0$ and any $x\in \X(\br)$, the map 
$$\iota_x^*R\iota^!\bz(n)\rightarrow \iota_x^*R\iota^!R\underline{\mathrm{Hom}}(\bz'(e-n),\D_{\X_{\infty}}[-2e])$$
can be identified with the identity
$$\tau_{\leq e-n-2}C_*(G_{\br},\bz(e-n))[-e]= \tau_{\leq e-n-2}C_*(G_{\br},\bz(e-n))[-e]$$
by Prop. \ref{iota1} and Prop. \ref{iota2}.

For $e-n\leq0$ and any $x\in \X(\br)$, the map 
$$\iota_x^*R\iota^!\bz(n)\rightarrow \iota_x^*R\iota^!R\underline{\mathrm{Hom}}(\bz'(e-n),\D_{\X_{\infty}}[-2e])$$
can be identified with the identity
$$\tau^{\leq n-e}R\Gamma(G_{\br},\bz(n-e))[-e]= \tau^{\leq n-e}R\Gamma(G_{\br},\bz(n-e))[-e]$$
by Prop. \ref{iota1} and Prop. \ref{iota3}.

The result follows since the family of functors $\{\iota_x^*,x\in\X(\br)\}$ is conservative.
\end{proof}


\subsection{Comparison with $R\Gamma(\X(\bc),\bz(n))$}

Recall that we define $G_\br$-equivariant sheaves 
$$\bz(n):=(2i\pi)^n\bz\subset \bq(n):=(2i\pi)^n\bq \subset \br(n):=(2i\pi)^n\br\subset\bc$$
on $\X(\bc)$.  We abbreviate $C^*:=R\Hom(C,\bq)$ for a complex of $\bq$-vector spaces $C$ and let $C^\pm$ be the image of the idempotent $(\sigma\pm 1 )/2$ if $C$ carries a $G_\br=\{1,\sigma\}$-action. Recall that 
$$R\Gamma_W(\X_{\infty},\bz(n)):=R\Gamma(\X_{\infty},i^*_{\infty}\bz(n))$$ 
and that 
$i^*_{\infty}\bz(n)\otimes\bq \cong \pi_*\bq(n)\cong R\pi_*\bq(n)$ in $\mathrm{Sh}(\X_\infty)$. We therefore have isomorphisms
$$R\Gamma_W(\X_{\infty},\bz(n))_\bq  \simeq R\Gamma(\X_{\infty},R\pi_*\bq(n))\simeq R\Gamma(G_{\br}; \X(\bc),\bq(n))\simeq R\Gamma(\X(\bc),\bq(n))^+$$
and combining this with Poincar\'e duality
\begin{equation}  R\Gamma(\mathcal{X}(\mathbb{C}),\bq(r)) \otimes R\Gamma(\mathcal{X}(\mathbb{C}),\bq(e-r))\xrightarrow{\cup} R\Gamma(\mathcal{X}(\mathbb{C}),\bq(e))\xrightarrow{\mathrm{Tr}}\bq[-2e]\label{bettipd}\end{equation}
on the $2e$-manifold $\X(\bc)$ we obtain an isomorphism
\begin{equation}R\Gamma_W(\mathcal{X}_{\infty},\mathbb{Z}(d-n))^*_\bq \simeq R\Gamma(\mathcal{X}(\mathbb{C}),\bq(d-n))^{*, +} \simeq R\Gamma(\mathcal{X}(\mathbb{C}),\bq(n-1))^{+}[-2e] \notag\end{equation}
using $e=d-1$. There is also a tautological isomorphism $\tau$ induced by multiplication by $2\pi i$ in the sense that the diagram
\begin{equation}\begin{CD} R\Gamma(\mathcal{X}(\mathbb{C}),\bq(n-1))^+@>>> R\Gamma(\mathcal{X}(\mathbb{C}),\bc)\\
@V\sim V\tau V @V\sim V\cdot 2\pi i V\\
R\Gamma(\mathcal{X}(\mathbb{C}),\bq(n))^-@>>> R\Gamma(\mathcal{X}(\mathbb{C}),\bc)\end{CD}\label{noncan}\end{equation}
commutes. 
Combining the previous isomorphisms we obtain an isomorphism
\begin{align}  &\left(\mydet_\bz R\Gamma_W(\mathcal{X}_{\infty},\mathbb{Z}(n))\otimes  \mathrm{det}^{-1}_{\mathbb{Z}}R\Gamma_{W}(\mathcal{X}_{\infty},\mathbb{Z}(d-n))\right)_\bq \notag \\
 \simeq  \,\,&\mydet_\bq \left( R\Gamma(\mathcal{X}(\mathbb{C}),\bq(n))^+\oplus R\Gamma(\mathcal{X}(\mathbb{C}),\bq(n-1))^+\right) \label{firstiso} \\
 \simeq  \,\,&\mydet_\bq \left( R\Gamma(\mathcal{X}(\mathbb{C}),\bq(n))^+\oplus R\Gamma(\mathcal{X}(\mathbb{C}),\bq(n))^-\right) \notag\\
 \simeq \,\,&\mydet_\bq R\Gamma(\mathcal{X}(\mathbb{C}),\bq(n))\notag\\
 \simeq\,\, &(\mydet_\bz R\Gamma(\mathcal{X}(\mathbb{C}),\bz(n)))_\bq\notag
\end{align}
which we denote by $\lambda_B$. 

\begin{cor}
We have
\begin{multline*} \lambda_B\left( \mathrm{det}_{\bz}R\Gamma_W(\X_{\infty},\bz(n))\otimes \mathrm{det}^{-1}_{\bz}R\Gamma_W(\X_{\infty},\bz(d-n))\right)\\ = \mathrm{det}_{\bz}R\Gamma(\X(\bc),\bz(n)) \otimes \mathrm{det}^{(-1)^n}_{\bz}R\Gamma(\X(\br),\bz/2\bz)\end{multline*}
\label{lamba-betti-iso}\end{cor}

\begin{proof}
We write $G_{\br}=\{1,\sigma\}$. We have an exact sequence of $\bz[G_{\br}]$-modules
$$0\rightarrow \bz\cdot(\sigma-1)\rightarrow \bz[G_{\br}]\stackrel{\epsilon}{\rightarrow} \bz\rightarrow 0$$
where $\epsilon$ is the augmentation map. We have an isomorphism of $\bz[G_{\br}]$-modules
$(\sigma-1)\cdot\bz\simeq (2i\pi)\bz$ which maps $(\sigma-1)$ to $(2i\pi)$. We write $\bz(n):=(2i\pi)^n\bz$, so that we have an exact sequence of $\bz[G_{\br}]$-modules
\begin{equation}\label{exsq}
0\rightarrow \bz(1)\rightarrow \bz[G_{\br}]\stackrel{\epsilon}{\rightarrow} \bz\rightarrow 0.
\end{equation}
We denote by
$$p:\mathrm{Sh}(G_{\br},\X(\bc))\rightarrow \mathrm{Sh}(G_{\br},\X_{\infty})$$
the morphism of topoi induced by the equivariant continuous map $p:\X(\bc)\rightarrow \X_{\infty}$, where $G_{\br}$ acts trivially on $\X_{\infty}$. The category of abelian sheaves on $\mathrm{Sh}(G_{\br},\X_{\infty})$ is equivalent to the category of sheaves of $\bz[G_{\br}]$-modules over $\X_{\infty}$.
For any sheaf $\mathcal{F}$ of $\bz[G_{\br}]$-modules over $\X_{\infty}$, and any $\bz[G_{\br}]$-module $M$, we define
$$R\underline{\mathrm{Hom}}_{\mathrm{Sh}(G_{\br},\X_{\infty})}(M,\mathcal{F})$$
where $M$ is seen as a constant sheaf of $\bz[G_{\br}]$-modules over $\X_{\infty}$. We have
$$R\pi_*\bz(n)\simeq R\underline{\mathrm{Hom}}_{\mathrm{Sh}(G_{\br},\X_{\infty})}(\bz,Rp_*\bz(n)).$$
Moreover the functor
$$\appl{\mathrm{Ab}(G_{\br},\X_{\infty})}{\mathrm{Ab}(G_{\br},\X_{\infty})}{\mathcal{F}}{\mathcal{F}(1):=\mathcal{F}\otimes_{\bz}\bz(1)}$$
is an equivalence of abelian categories with quasi-inverse $(-)\otimes_{\bz}\bz(-1)$. In particular we have
\begin{eqnarray*}
R\pi_*\bz(n-1)&\simeq&R\underline{\mathrm{Hom}}_{\mathrm{Sh}(G_{\br},\X_{\infty})}(\bz,Rp_*\bz(n-1))\\
&\simeq&R\underline{\mathrm{Hom}}_{\mathrm{Sh}(G_{\br},\X_{\infty})}(\bz(1),(Rp_*\bz(n-1))(1))\\
&\simeq & R\underline{\mathrm{Hom}}_{\mathrm{Sh}(G_{\br},\X_{\infty})}(\bz(1),Rp_*\bz(n)).
\end{eqnarray*}
Finally, we have
$$p_*\bz(n)\simeq Rp_*\bz(n)\simeq R\underline{\mathrm{Hom}}_{\mathrm{Sh}(G_{\br},\X_{\infty})}(\bz[G_{\br}],Rp_*\bz(n)).$$
Therefore, (\ref{exsq}) induces an exact triangle
$$R\pi_*\bz(n)\rightarrow Rp_*\bz(n)\rightarrow R\pi_*\bz(n-1)$$
and an exact diagram:
\[ \xymatrix{
Rp_*\bz(n)\ar[d]\ar[r]^{} &i^*_{\infty}\bz(n-1)\ar[d]^{} \ar[r]^{}&i^*_{\infty}\bz(n)[1]\ar[d] \\
Rp_*\bz(n)\ar[d]\ar[r]^{} &R\pi_*\bz(n-1)\ar[d]^{} \ar[r]^{}&R\pi_*\bz(n)[1]\ar[d] \\
0\ar[r]^{}&\tau^{>n-1}R\widehat{\pi}_*\bz(n-1) \ar[r]^{}& (\tau^{>n}R\widehat{\pi}_*\bz(n))[1] \\
}
\]
In particular, there is an exact triangle
$$ i^*_{\infty}\bz(n)\rightarrow Rp_*\bz(n)\rightarrow i^*_{\infty}\bz(n-1)$$
hence
\begin{equation}\label{comp1}
R\Gamma(\X_{\infty},\bz(n))\rightarrow R\Gamma(\X(\bc),\bz(n))\rightarrow R\Gamma(\X_{\infty},\bz(n-1)).
\end{equation}
Moreover, we have the duality equivalence
\begin{equation}\label{comp2}
R\Gamma(\X_{\infty},\bz(d-n))\stackrel{\sim}{\rightarrow} R\mathrm{Hom}(R\Gamma(\X_{\infty},\bz'(n-1)),\bz[-2e]).
\end{equation}
Finally, we have the following exact diagram
\[ \xymatrix{
\iota_{*}\bz/2\bz_{\X(\br)}[-n]\ar[d]\ar[r]^{} &\bz'(n-1)\ar[d]^{} \ar[r]^{}&i^*_{\infty}\bz(n-1)\ar[d] \\
0\ar[d]\ar[r]^{} &R\pi_*\bz(n-1)\ar[d]^{} \ar[r]^{}&R\pi_*\bz(n-1)\ar[d] \\
\iota_{*}\bz/2\bz_{\X(\br)}[-n+1]\ar[r]^{}&\tau^{\geq n-1}R\widehat{\pi}_*\bz(n-1) \ar[r]^{}& \tau^{>n-1}R\widehat{\pi}_*\bz(n-1) \\
}
\]
where $\bz/2\bz_{\X(\br)}$ is the constant sheaf $\bz/2\bz$ on $\X(\br)$, hence an exact triangle
\begin{equation}\label{comp3}
R\Gamma(\X(\br),\bz/2\bz)[-n]\rightarrow R\Gamma(\X_{\infty},\bz'(n-1)) \rightarrow R\Gamma(\X_{\infty},\bz(n-1)).
\end{equation}
Then (\ref{comp1}), (\ref{comp2}) and (\ref{comp3}) induce the following canonical isomorphisms:
\begin{eqnarray*}
&&\mathrm{det}_{\bz}R\Gamma(\X_{\infty},\bz(n))\otimes \mathrm{det}^{-1}_{\bz}R\Gamma(\X_{\infty},\bz(d-n))\\
&\simeq& \mathrm{det}_{\bz}R\Gamma(\X_{\infty},\bz(n))\otimes \mathrm{det}_{\bz}R\Gamma(\X_{\infty},\bz'(n-1))\\
&\simeq& \mathrm{det}_{\bz}R\Gamma(\X_{\infty},\bz(n))\otimes \mathrm{det}_{\bz}R\Gamma(\X_{\infty},\bz(n-1))\otimes \mathrm{det}_{\bz}R\Gamma(\X(\br),\bz/2\bz)[-n]\\
&\simeq& \mathrm{det}_{\bz}R\Gamma(\X(\bc),\bz(n)) \otimes \mathrm{det}_{\bz}R\Gamma(\X(\br),\bz/2\bz)[-n]\\
&\simeq & \mathrm{det}_{\bz}R\Gamma(\X(\bc),\bz(n)) \otimes \mathrm{det}^{(-1)^n}_{\bz}R\Gamma(\X(\br),\bz/2\bz).
\end{eqnarray*}
\end{proof}

We now introduce some notation: we set
$$d_+(\X,n):= \sum_{i\in\bz}(-1)^{i} \mathrm{dim}_{\bq}H^{i}(\X(\bc),\bq(n))^+$$
and
$$d_-(\X,n):= \sum_{i\in\bz}(-1)^{i} \mathrm{dim}_{\bq}H^{i}(\X(\bc),\bq(n))^-.$$
If $Z$ is a manifold and $F$ a field, we set
$$\chi(Z,F):=\sum_{i\in\bz}(-1)^{i} \mathrm{dim}_{F}H^{i}(Z,F).$$ 

\begin{defn} For a perfect complex of abelian groups $C$ with finite cohomology groups we denote by
\[ \chi^\times(C)=\prod_{i\in\bz}|H^i(C)|^{(-1)^i}\]
its multiplicative Euler characteristic.
\label{chidef}\end{defn}

\begin{prop} We have
$$\chi^{\times}(R\Gamma(\X(\br),\bz/2\bz)[-n])=2^{d_+(\X,n)-d_-(\X,n)}.$$
\label{real-eulerchar}\end{prop}
\begin{proof}
We have 
$$d_+(\X,n)=d_-(\X,n-1)=d_+(\X,n-2)$$ 
hence 
$$d_{\pm}(\X,n)=(-1)^n\cdot d_{\pm}(\X,0).$$
We obtain
$$2^{d_+(\X,n)-d_-(\X,n)}=(2^{d_+(\X,0)-d_-(\X,0)})^{(-1)^n}.$$
Similarly, we have
$$\chi^{\times}(R\Gamma(\X(\br),\bz/2\bz)[-n]):=\chi^{\times}(R\Gamma(\X(\br),\bz/2\bz))^{(-1)^n},$$
hence it is enough to show the result for $n=0$. In view of Lemma \ref{Lefschetz} and Lemma \ref{chi}, we have
\begin{eqnarray*}
d_+(\X,0)-d_-(\X,0)&=&\sum_{i\in\bz}(-1)^{i}\cdot \left(\mathrm{dim}_{\bq}H^{i}(\X(\bc),\bq)^+-\mathrm{dim}_{\bq}H^{i}(\X(\bc),\bq)^-\right)\\
&=&\sum_{i\in\bz}(-1)^{i}\cdot \mathrm{Tr}\left(\sigma\mid H^{i}(\X(\bc),\bq)\right)\\
&=&\chi(\X(\br),\bq)\\
&=&\chi(\X(\br),\mathbb{F}_2).
\end{eqnarray*}
Hence the result follows from
$$\chi^{\times}(R\Gamma(\X(\br),\bz/2\bz))=2^{\chi(\X(\br),\mathbb{F}_2)}.$$
\end{proof}

\begin{lem}\label{Lefschetz}
Let $Y$ be a compact orientable manifold with an involution $\sigma$ whose fixed points form a closed submanifold $Z$. Then we have 
$$\sum_{i\in\bz}(-1)^{i}\cdot \mathrm{Tr}\left(\sigma\mid H^{i}(Y,\bq)\right)=\chi(Z,\bq).$$
\end{lem}
\begin{proof}
Let $G_{\br}:=\{1,\sigma\}$. If $C$ is a perfect complex of $\bq$-vector spaces with $G_{\br}$-action, we set
$$\mathrm{Tr}\left(\sigma\mid C\right):=\sum_{i\in\bz}(-1)^{i}\cdot \mathrm{Tr}\left(\sigma\mid H^{i}(C)\right).$$
Let $Y^{\circ}:=Y-Z$. The exact triangle
$$R\Gamma_c(Y^{\circ},\bq)\rightarrow R\Gamma(Y,\bq)\rightarrow R\Gamma(Z,\bq)$$
gives
\begin{eqnarray*}
\mathrm{Tr}\left(\sigma\mid R\Gamma(Y,\bq)\right)&=& \mathrm{Tr}\left(\sigma\mid R\Gamma_c(Y^{\circ},\bq)\right)+\mathrm{Tr}\left(\sigma\mid R\Gamma(Z,\bq)\right)\\
&=&\mathrm{Tr}\left(\sigma\mid R\Gamma_c(Y^{\circ},\bq)\right)+\chi(Z,\bq)
\end{eqnarray*}
since $\sigma$ acts trivially on $Z$ hence on $R\Gamma(Z,\bq)$. Therefore the result follows from
\begin{eqnarray*}
\mathrm{Tr}\left(\sigma\mid R\Gamma_c(Y^{\circ},\bq)\right)&:=& \sum_{i\in\bz}(-1)^{i}\cdot \mathrm{Tr}\left(\sigma\mid H^{i}_c(Y^{\circ},\bq)\right)\\
&=& \sum_{i\in\bz}(-1)^{i}\cdot \mathrm{Tr}\left(\sigma\mid H^{i}_c(Y^{\circ},\bq)^*\right)\\
&=& \sum_{i\in\bz}(-1)^{i}\cdot \mathrm{Tr}\left(\sigma\mid H^{d-i}(Y^{\circ},\bq)\right)\\
&=&(-1)^{d} \sum_{i\in\bz}(-1)^{i}\cdot \mathrm{Tr}\left(\sigma\mid H^{i}(Y^{\circ},\bq)\right)\\
&=&0.
\end{eqnarray*}
where we use Poincar\'e duality and the Lefschetz fixed point theorem. Here $d=\mathrm{dim}(Y)$.
\end{proof}

\begin{lem}\label{chi}
Let $Z$ be a topological space which is homotopy equivalent to a finite $CW$-complex. Then we have 
$$\chi(Z,F)=\chi(Z,F')$$
for any pair of fields $F,F'$.
\end{lem}
\begin{proof}
The complex $R\Gamma(Z,\bz)$ is quasi-isomorphic to a strictly perfect complex of abelian groups $C^*$ and we have
$$\sum_{i\in\bz}(-1)^{i}\mathrm{rank}_{\bz}C^{i}=\sum_{i\in\bz}(-1)^{i}\mathrm{dim}_{F}(C^{i}\otimes_{\bz} F)= \sum_{i\in\bz}(-1)^{i}\mathrm{dim}_{F}H^{i}(C^*\otimes_{\bz} F)=\chi(Z,F)$$
for any field $F$. The result follows.
 \end{proof}

Combining Corollary \ref{lamba-betti-iso} with Prop. \ref{real-eulerchar} we obtain.

\begin{prop}
We have
\begin{multline*} \lambda_B\left( \mathrm{det}_{\bz}R\Gamma_W(\X_{\infty},\bz(n))\otimes \mathrm{det}^{-1}_{\bz}R\Gamma_W(\X_{\infty},\bz(d-n))\right)\\ = \mathrm{det}_{\bz}R\Gamma(\X(\bc),\bz(n)) \cdot 2^{d_-(\X,n)-d_+(\X,n)}\end{multline*}
\label{final-bettiprop}\end{prop}

\begin{proof} Note that if $C$ is as in Definition \ref{chidef} then $$\mydet_\bz C=\bz\cdot \chi^\times(C)^{-1}$$ under the canonical isomorphism
$$\mydet_\bq C_\bq\cong \bq$$ arising from the acyclicity of $C_\bq$.
\end{proof}

\section{Duality for derived de Rham cohomology and the Bloch conductor} \label{bloch}

In this section $\X$ is a regular scheme of dimension $d$, proper and flat over $\Spec(\bz)$. We denote by 
\begin{equation} L_{\X/\bz}\cong \Omega_{\X/\bz}[0]\label{cotangent}\end{equation}
the cotangent complex of $\X$ over $\bz$, a perfect complex of $\co_{\X}$-modules cohomologically concentrated in degree $0$. For any $r\in\bz$ we let
\[ L\wedge ^rL_{\X/\bz}\cong L\wedge^r \Omega_{\X/\bz}[0]\]
be the $r$-th derived exterior power of $L_{\X/\bz}$ \cite{Illusie71}[4.2.2.6] which is again a perfect complex of $\co_\X$-modules.  By definition $L\wedge ^r L_{\X/\bz}=0$ for $r<0$ but $L\wedge ^r L_{\X/\bz}$ is in general nonzero for $r>d-1=\rank_{\co_\X}\Omega_{\X/\bz}$.

\subsection{Coherent duality for $L\wedge ^r L_{\X/\bz}$} This subsection is a review of material from \cite{tsaito88}, \cite{ks04} and \cite{tsaito04} in the context of our global arithmetic scheme $\X$. The key result is Thm. \ref{lichtenbaum} which is an immediate translation of \cite{tsaito04}[Cor. 4.9] to our context.

\begin{lem} There is a canonical map 
\begin{equation} L\wedge ^{d-1}L_{\X/\bz}\to \mydet_{\co_\X}L_{\X/\bz}\cong \omega_{\X/\bz}\label{can}\end{equation}
where $\omega_{\X/\bz}$ is the relative dualizing sheaf. Hence we get induced maps
\begin{equation} L\wedge ^r L_{\X/\bz}\otimes^L L\wedge ^{d-1-r}L_{\X/\bz}\to  L\wedge ^{d-1}L_{\X/\bz}\to  \omega_{\X/\bz}\notag\end{equation}
and 
\begin{equation} L\wedge ^r L_{\X/\bz}\to \underline{R\Hom}(L\wedge ^{d-1-r}L_{\X/\bz},\omega_{\X/\bz})\label{duality}\end{equation}
in the derived category of coherent sheaves on $\X$.
\end{lem}

\begin{proof} The multiplicative structure on derived exterior powers will be briefly recalled in the proof of Prop. \ref{drduality} below, so it remains to show the existence of (\ref{can}). Assume first there is a closed embedding $i:\X\to P$ of $\X$ into a smooth $\bz$-scheme $P$ with ideal sheaf $\mathcal{I}$.  The exact sequence of coherent sheaves on $\X$
\[ 0\to \mathcal{I}/\mathcal{I}^2\to i^*\Omega_{P/\bz}\to \Omega_{\X/\bz}\to 0\]
can be viewed as a realization of (\ref{cotangent}) as a strictly perfect complex since $\mathcal{I}/\mathcal{I}^2$ and $i^*\Omega_{P/\bz}$ are locally free of ranks $n-d+1$ and $n$, respectively, where $n$ is the relative dimension of $P$ over $\bz$. The natural map
$$ \wedge^{d-1}\Omega_{\X/\bz}\otimes \wedge^{n-d+1}(\mathcal{I}/\mathcal{I}^2) \to \wedge^n i^*\Omega_{P/\bz} $$
has adjoint
$$ \wedge^{d-1}\Omega_{\X/\bz}\to \underline{\Hom}(\wedge^{n-d+1}(\mathcal{I}/\mathcal{I}^2),\wedge^n i^*\Omega_{P/\bz}) =: \omega_{\X/\bz}^P$$
and combined with the natural map
$$ L\wedge ^{d-1}L_{\X/\bz}\to \mathcal{H}^0(L\wedge ^{d-1}L_{\X/\bz})\cong \wedge^{d-1}\Omega_{\X/\bz}$$
we obtain a morphism (\ref{can})$^P$ depending on $i:\X\to P$. If $i':\X\to P'$ is another embedding into a smooth $\bz$-scheme $P'$ an isomorphism 
$$\epsilon^{P',P}:  \omega_{\X/\bz}^P\xrightarrow{\sim} \omega_{\X/\bz}^{P'} $$
was constructed in \cite{beresrue10}[A.2]  which satisfies the usual cocycle condition in the presence of a third embedding $i''$. Since embeddings into smooth schemes always exist Zariski locally on $\X$ the cocycle condition implies that one can define $\omega_{\X/\bz}$ by glueing the locally defined  $\omega_{\X/\bz}^P$. It remains to show that likewise the locally obtained maps (\ref{can})$^P$ glue to a global map (\ref{can}). By considering the fibre product $P'':=P\times_{\Spec(\bz)} P'$  the construction of $\epsilon^{P',P}$  can be reduced to the case where there exists a smooth morphism $u:P'\to P$ over $\Spec(\bz)$ and under $\X$. Namely one defines 
$$ \epsilon^{P',P}:=\epsilon^{P'',P'}(q')^{-1}\circ \epsilon^{P'',P}(q)$$
where $q':P''\to P'$ and $q:P''\to P$ are the projections and 
$$\epsilon^{P',P}(u):  \omega_{\X/\bz}^P\xrightarrow{\sim} \omega_{\X/\bz}^{P'} $$
depends on $u$. More precisely, 
$\epsilon^{P',P}(u)$ is defined by
the commutative diagram with exact rows and columns
\[
\begin{CD}  @. 0 @. 0 @.\\
@. @VVV @VVV @.\\
0 @>>>  \mathcal{I}/\mathcal{I}^2 @>>> i^*\Omega_{P/\bz} @>>> \Omega_{\X/\bz} @>>> 0\\
@. @VVV @VVV \Vert @. @.\\
0 @>>>  \mathcal{I}'/(\mathcal{I}')^2 @>>> i'^*\Omega_{P'/\bz} @>>> \Omega_{\X/\bz} @>>> 0\\
@. @VVV @VVV  @. @.\\
@. i'^*\Omega_{P'/P} @=  i'^*\Omega_{P'/P} @. \\
@. @VVV @VVV  @. @.\\
@. 0 @. 0 @.
\end{CD}\]
where the columns are the transitivity triangles of the cotangent complex for $\X\to P'\xrightarrow{u} P$ and $P'\xrightarrow{u} P\to\Spec(\bz)$, respectively, and we refer to 
\cite{beresrue10}[(A.2.2)] for the precise sign conventions. The above commutative diagram induces a commutative diagram
$$\begin{CD} \wedge^{d-1}\Omega_{\X/\bz} @>>>  \omega_{\X/\bz}^P\\
\Vert @. @VV \epsilon^{P',P}(u) V\\
 \wedge^{d-1}\Omega_{\X/\bz} @>>>  \omega_{\X/\bz}^{P'},
 \end{CD} $$
so that  (\ref{can})$^P$ is indeed compatible with the isomorphisms  $\epsilon^{P',P}(u) $ and therefore also with the  isomorphisms $\epsilon^{P',P}$.

\end{proof}

\begin{defn}\label{Bloch}
The {\em Bloch conductor} of the arithmetic scheme $\X$ is the positive integer
$$A(\X):= \prod_p p^ {(-1)^{d-1}d_p }$$
where the product is over all prime numbers $p$, $d_p:=\deg c_{d,\X_{\bF_p}}^\X(\Omega_{\X/\bz})\in\bz$ and 
$$c_{d,\X_{\bF_p}}^\X(\Omega_{\X/\bz})\in CH_0(\X_{\bF_p})$$ is a localized Chern class introduced in \cite{bloch87a}. 
\end{defn}
The Bloch conductor was introduced in \cite{bloch87a} and further studied in \cite{Bloch87},\cite{tsaito88},\cite{ks04},\cite{tsaito04}. The deepest result about the Bloch conductor is its equality with the Artin conductor, defined in terms of the $l$-adic cohomology of $\X$, in certain cases. This equality was proven for $d=2$ in \cite{bloch87a} and if $\X$ has everywhere semistable reduction in \cite{ks04}.  For general regular $\X$ it is conjectured but still open. The equality of the Bloch and the Artin conductor is important for establishing cases of Conjecture \ref {conj-funct-equation} via the Langlands correspondence but plays no role in this section. Here we only review the (slightly) more elementary results of \cite{tsaito88} and \cite{tsaito04} about $A(\X)$. Also note that our normalization of $A(\X)$ is different from these references so that $A(\X)$ equals the Artin conductor rather than its inverse.

The following theorem was proven by T. Saito in \cite{tsaito04}[Cor. 4.9]. The case $d=2$, $r=1$ is due to Bloch \cite{Bloch87}[Thm. 2.3] and the case $r\geq d-1$ can already be found in T. Saito's earlier article \cite{tsaito88}. We give some details of Saito's proof since the exposition in \cite{tsaito04} is rather short. 

\begin{thm} For any $r\in\bz$ let $C_{\X/\bz}^r$ be the mapping cone of (\ref{duality}), a perfect complex of $\co_{\X}$-modules. Then $R\Gamma(\X,C_{\X/\bz}^r)$ has finite cohomology and 
$$ \chi^\times R\Gamma(\X,C_{\X/\bz}^r)=A(\X)^{(-1)^{r}}$$
where $\chi^\times$ is the multiplicative Euler characteristic (see Definition \ref{chidef}).
\label{lichtenbaum}\end{thm}

\begin{proof} First note that over the open subset $\X^{sm}\subseteq \X$ where $\X\to\Spec(\bz)$ is smooth the complex $L\wedge ^r L_{\X/\bz}$ is concentrated in degree $0$ with cohomology the locally free sheaf $\Omega^r_{\X^{sm}/\bz}=\wedge^r\Omega_{\X^{sm}/\bz}$. The map (\ref{can}) is also an isomorphism over $\X^{sm}$. Hence, by linear algebra, the map (\ref{duality}) is an isomorphism over $\X^{sm}$ and $C_{\X/\bz}^r$ is supported in $\X\setminus \X^{sm}$. Since $\X_\bq\to \Spec(\bq)$ is smooth $\X\setminus \X^{sm}$ is contained in a finite union of closed fibres $\X_{\bF_p}$. By \cite{ks04}[Lemma 5.1.1] any point $x\in\X\setminus \X^{sm}$ has a Zariski open neighborhood $U\subseteq \X$ such that there exists a closed embedding
\[ U\to P\] 
where $P\to\Spec(\bz)$ is smooth of relative dimension $d$. The exact sequence
\begin{equation} 0\to N_{U/P}\to \Omega_{P/\bz}\otimes _{\co_P}\co_U\to \Omega_{U/\bz}\to 0\label{loc1}\end{equation}
then shows that $\Omega_{\X/\bz}$ can be locally generated by $d$ sections and that $\wedge^d \Omega_{\X/\bz}$ is locally monogenic. Following \cite{ks04}[Lemma 5.1.3] let 
\[ i:Z\to\X \] 
be the closed subscheme with support $\X\setminus \X^{sm}$ \cite{ks04}[Lemma 3.1.2] defined by the ideal sheaf $$\mathrm{Ann} \wedge^d \Omega_{\X/\bz}.$$
Then $i^*\wedge^d \Omega_{\X/\bz}$ is an invertible $\co_Z$-module by definition and hence $i^*\Omega_{\X/\bz}$ is locally free of rank $d$, as the $d$ generating sections have no relation on $Z$. It follows that 
$$Li^*\Omega_{\X/\bz}\vert_U = \left( N_{U/P}\otimes_{\co_U}\co_{U\cap Z} \xrightarrow{0}\Omega_{P/\bz}\otimes _{\co_P}\co_{U\cap Z}\right) $$
and hence that
$$\mathcal{L}:=L^1i^*\Omega_{\X/\bz}$$
is an invertible $\co_Z$-module. 

\begin{lem} The coherent sheaves $\mathcal{H}^i(C_{\X/\bz}^r)$ are $\co_Z$-modules and there are canonical isomorphisms
\begin{equation}\mathcal{L}\otimes_{\co_Z} \mathcal H^i(C_{\X/\bz}^r)\cong \mathcal H^{i-1}(C_{\X/\bz}^{r+1}) \label{globalhmap}\end{equation}
for any $i,r\in\bz$.
\end{lem} 

\begin{proof} We follow the proof of \cite{tsaito88}[Prop. 1.7] where the case $r\geq d-1$ is treated, see also \cite{ks04}[Lemma 2.4.2]. Recall that 
$$ L_{\X/\bz} \vert_U\cong \Omega_{U/\bz}[0] $$ 
is represented by the strictly perfect complex (\ref{loc1}) where the conormal bundle $N_U:=N_{U/P}$ is invertible and  
$E_U:=\Omega_{P/U}\otimes_{\co_P}\co_U$ is a vector bundle of rank $d$. For $r\geq 0$ we have isomorphisms
\begin{align} L\wedge^r L_{\X/\bz} \vert_U \cong & L\wedge^r \left(N_U \xrightarrow{v} E_U\right) \notag\\
\cong & \Gamma^rN_U \to \Gamma^{r-1}N_U\otimes E_U \to \Gamma^{r-2}N_U\otimes \wedge^2 E_U\to \cdots\to \wedge^r E_U\notag \\
\cong & N_U^{\otimes r} \to N_U^{\otimes r-1}\otimes E_U \to N_U^{\otimes r-2}\otimes \wedge^2 E_U\to \cdots\to \wedge^r E_U \label{negpart}
\end{align}
where $\Gamma^i$ denotes the divided power functor and $\Gamma^i N_U\cong N_U^{\otimes i}$ since $N_U$ is invertible. The differential is given by 
\begin{equation}x'\otimes x\otimes y\in N_U^{\otimes i-1}\otimes N_U\otimes\wedge^{r-i}E_U\ \mapsto\  x'\otimes v(x)\wedge y\in  N_U^{\otimes i-1}\otimes\wedge^{r-i+1}E_U\label{ldif}\end{equation}
on local sections. This computation of the derived exterior powers of a strictly perfect  two-term complex  goes back to Illusie \cite{Illusie71}[4.3.1.3] and is also recalled in \cite{ks04}[1.2.7.2].  From this description it is clear that there is an identity of complexes 
\begin{equation}N_U\otimes L\wedge^r \left(N_U \xrightarrow{v} E_U\right) = \left(\sigma^{<0}  L\wedge^{r+1} \left(N_U \xrightarrow{v} E_U\right)\right)[-1]\label{twist1}\end{equation}
where $\sigma^{<0}$ refers to the naive truncation.
Similarly we find
\begin{align} 
&\underline{R\Hom}(L\wedge ^{d-1-r}L_{\X/\bz},\omega_{\X/\bz})\vert_U \cong  \underline{\Hom}( L\wedge^{d-1-r} \left(N_U \xrightarrow{v} E_U\right), K_U) \notag\\
\cong &  \underline{\Hom}(\wedge^{d-1-r} E_U, K_U)\to\cdots \to 
 \underline{\Hom}( N_U^{\otimes i}\otimes \wedge^{d-1-r-i}E_U, K_U)\to\cdots \label{pospart}
\end{align}
where
$$ K_U:= N_U^{-1}\otimes \wedge^d E_U \cong \omega_{\X/\bz}\vert_U .$$
Using the canonical isomorphism
\begin{multline}N_U\otimes \underline{\Hom}( N_U^{\otimes i}\otimes \wedge^{d-1-r-i}E_U, K_U)\\
\cong  \underline{\Hom}( N_U^{\otimes i-1}\otimes \wedge^{d-1-(r+1)-(i-1)}E_U, K_U)\label{twist4}\end{multline}
we find a canonical isomorphism of complexes
\begin{multline}\sigma^{>0}N_U\otimes \underline{\Hom}( L\wedge^{d-1-r} \left(N_U \xrightarrow{v} E_U\right), K_U) \\ \cong  \underline{\Hom}( L\wedge^{d-1-(r+1)} \left(N_U \xrightarrow{v} E_U\right), K_U)[-1].\label{twist2} \end{multline}
The complex $C^r_{\X/\bz}\vert_U $ is obtained by splicing together (\ref{negpart}) placed in degrees $\leq -1$ with (\ref{pospart}) placed in degrees $\geq 0$ via the map
$$ \phi_r:\wedge^r E_U\to \underline{\Hom}(\wedge^{d-1-r} E_U, K_U) \cong \underline{\Hom}(N_U\otimes \wedge^{d-1-r} E_U, \wedge^d E_U)$$
dual to 
$$  \wedge^r E_U \otimes N_U\otimes \wedge^{d-1-r} E_U \to \wedge^d E_U;\quad y\otimes x\otimes y'\mapsto v(x)\wedge y\wedge  y' .$$
Denoting by $\psi$ the canonical isomorphism
\[ \psi: \wedge^{r+1} E_U\cong \underline{\Hom}(\wedge^{d-1-r} E_U, \wedge^d E_U) \cong N_U\otimes\underline{\Hom}(\wedge^{d-1-r} E_U, K_U)\]
we have a commutative diagram
\begin{equation}\minCDarrowwidth1em
\begin{CD} 
N_U\otimes\wedge^r E_U @>(\ref{ldif})>> \wedge^{r+1} E_U @>\phi_{r+1}>> \underline{\Hom}( \wedge^{d-2-r} E_U, K_U) \\
\Vert @. @V\sim V\psi V @V\sim V(\ref{twist4})^{-1} V\\
N_U\otimes\wedge^r E_U @>\id\otimes \phi_r>> N_U\otimes\underline{\Hom}(\wedge^{d-1-r} E_U, K_U) @>>>N_U\otimes \underline{\Hom}(N_U\otimes \wedge^{d-2-r} E_U, K_U) 
\end{CD}
\label{twist3}\end{equation}
as one verifies easily on local sections. Combining (\ref{twist1}), (\ref{twist2}) and (\ref{twist3}) we obtain a canonical isomorphism
\begin{equation} N_U\otimes C^r_{\X/\bz}\vert_U \cong C^{r+1}_{\X/\bz}\vert_U [-1] . \label{twist5}\end{equation}
As in \cite{tsaito88}[(1.6.1)] one has an isomorphism 
\[ C^{d-1}_{\X/\bz}\vert_U \cong  K_U\otimes \mathrm{Kos}(E_U^* \otimes N_U\xrightarrow{v^*\otimes\id} N_U^*\otimes N_U\cong\co_U)\] 
where $\mathrm{Kos}(P\to A)$ denotes the Koszul algebra associated to a $A$-module homomorphism $P\to A$ where $P$ is finitely generated projective over $A$. Using the fact that $H^i(\mathrm{Kos}(P\to A))$ is a module over the ring $H^0(\mathrm{Kos}(P\to A))$ \cite{stacks}[15.28.6] one deduces that all coherent sheaves 
$\mathcal H^i( C^{d-1}_{\X/\bz}\vert_U )$ are modules over $H^0(\mathrm{Kos})\cong \co_{U\cap Z}$.  Using  (\ref{twist5})  and the fact that $N_U$ is invertible we deduce that all  coherent sheaves 
$\mathcal H^i( C^{r}_{\X/\bz}\vert_U )$ are modules over $\co_{U\cap Z}$, and an isomorphism
\begin{equation}\left(\mathcal L\otimes_{\co_Z}\mathcal H^i(C^r_{\X/\bz}) \right) |_U \cong \mathcal H^{i-1}(C^{r+1}_{\X/\bz}) \vert_U \label{hmap}\end{equation}
whose construction a priori depends on the choice of $U\to P$. However, as in the proof of \cite{tsaito88}[(1.7.2)]  one shows that for a different embedding $U\to P'$, leading to a  different strictly perfect resolution $N'_U\to E_U'$ of $L_{\X/\bz}\vert_U$, one has a quasi-isomorphism 
$$ g: \left(N_U'\to E_U' \right)\to \left(N_U\to E_U\right)$$
unique up to homotopy, inducing quasi-isomorphisms 
$$ g^r: L\wedge^r \left(N_U'\to E_U' \right)\to L\wedge^r \left(N_U\to E_U\right)$$
for all $r$, unique up to homotopy, which commute with the isomorphisms (\ref{twist1}), (\ref{twist2}) and (\ref{twist3}). Hence (\ref{hmap}) is in fact independent of the choice of $U\to P$ which also implies that the local isomorphisms (\ref{hmap}) glue to the global isomorphism (\ref{globalhmap}).
\end{proof} 

Since $\X\to\Spec(\bz)$ is proper and $C^r_{\X/\bz}$ is a perfect complex complex of $\co_\X$-modules $R\Gamma(\X, C^r_{\X/\bz})$ is a perfect complex of abelian groups. It has a finite filtration with subquotients 
$$R\Gamma(\X, \mathcal H^i(C^r_{\X/\bz})[-i]) \cong R\Gamma(Z, \mathcal H^i(C^r_{\X/\bz})[-i])$$
which are perfect complexes of abelian groups with torsion cohomology, as $Z$ is supported in a finite union of closed fibres $\X_{\bF_p}$. Hence $R\Gamma(\X, C^r_{\X/\bz})$ has finite cohomology. We can view $\chi^\times$ as a homomorphism
$$ \chi^\times: G(Z) \to K_0(\bz;\bq)\cong \bq^{\times,>0};\quad [\F] \mapsto [R\Gamma(Z,\F)]$$
where $G(Z)$ is the Grothendieck group of the category of coherent sheaves on $Z$ and $K_0(\bz;\bq)$ is the Grothendieck group of the category of finite abelian groups (which is also the relative $K_0$ for the ring homomorphism $\bz\to\bq$). By \cite{ks04}[Lemma 5.1.3.3] one has $[\mathcal L\otimes_{\co_Z}\F]=[\F]$ in $G(Z)$ for any coherent sheaf $\F$ on $Z$.  Hence (\ref{globalhmap}) implies 
$$ \chi^\times R\Gamma(Z, \mathcal H^i(C^r_{\X/\bz})) =\chi^\times R\Gamma(Z, \mathcal H^{i-1}(C^{r+1}_{\X/\bz}))$$
and therefore 
$$\chi^\times R\Gamma(\X, C^r_{\X/\bz})=\chi^\times R\Gamma(\X, C^{r+1}_{\X/\bz})^{-1}$$
for any $r\in\bz$. On the other hand we have 
$$ \chi^\times R\Gamma(\X, C^d_{\X/\bz}) = \chi^\times R\Gamma(\X, L\wedge^dL_{\X/\bz}[1]) = A(\X)^{(-1)^{d}}$$
by \cite{tsaito88}[Prop. 2.3]. This finishes the proof of the theorem.
\end{proof}
  
\subsection{Duality for derived de Rham cohomology} Denote by 
$$\cdots \to F^{r+1}\to F^r\to \cdots \to R\Gamma_{dR}(\mathcal{X}/\mathbb{Z})=F^0=F^{-1}=\cdots $$ 
the Hodge filtration of (Hodge completed) derived de Rham cohomology and by $F^n/F^m$ the mapping cone of $F^m\to F^n$ for $m\geq n$. Since 
\begin{equation} F^r/F^{r+1}\cong R\Gamma(\X,L\wedge ^r L_{\X/\bz}[-r])\label{graded}\end{equation}
is a perfect complex of abelian groups, so are all $F^n/F^m$ for $m\geq n$. Denote by $C^*=R\Hom(C,\bz)$ the $\bz$-dual of a perfect complex of abelian groups.

\begin{prop} a) For $n\leq d$ there is a (Poincar\'e) duality map 
\begin{equation} \epsilon_n:F^n/F^d\to (R\Gamma_{dR}(\mathcal{X}/\mathbb{Z})/F^{d-n})^*[-2d+2] \label{pd}\end{equation}
satisfying
\begin{equation}\chi^\times \Cone(\epsilon_n)=A(\X)^{d-n}.\label{discrepancy}\end{equation}
b) In particular, the discriminant of the Poincar\'e duality pairing
\begin{equation} R\Gamma_{dR}(\X/\bz)/F^d\otimes_\bz^L R\Gamma_{dR}(\X/\bz)/F^d\rightarrow \bz[-2d+2]\label{disc}\end{equation}
has absolute value $A(\X)^d$.
\label{drduality}\end{prop}

\begin{rem} For $d=1$ we have $\X=\Spec(\co_F)$ and $A(\X)=|D_F|$, and b) reduces to the fact that the trace pairing
\[ \co_F\times \co_F \to \bz;\quad (a,b)\mapsto \trace(ab)\]
has discriminant $D_F$. For $d=2$ it was shown by Bloch in \cite{Bloch87}[Thm. 2] that the Poincar\'e duality pairing 
on the complex
$$ R\Gamma(\X,\co_{\X}\to \Omega_{\X/\bz})\cong R\Gamma_{dR}(\X/\bz)/F^2$$
has discriminant $\pm A(\X)^2$. For $d\geq 3$ it seems harder to describe the complex
$R\Gamma_{dR}(\X/\bz)/F^d$ more explicitly.
\end{rem}

\begin{rem} If $P$ is a perfect complex of abelian groups and $P\otimes_\bz^LP\to \bz[2\delta]$ is a pairing which induces an
isogeny $\phi:P\to P^*[2\delta]$ in the sense that $\cone(\phi)$ has finite cohomology groups, we obtain isomorphisms
$$ \mydet_\bz P^*\simeq \mydet_\bz P\otimes_\bz\mydet_\bz\cone(\phi)$$
and
$$ \mydet_\bz P\otimes \mydet_\bz P\simeq \mydet^{-1}_\bz\cone(\phi)$$
and hence a duality pairing on determinants
$$\langle \cdot,\cdot\rangle: \mydet_\bq P_\bq \otimes \mydet_\bq P_\bq \simeq \bq .$$ 
The discriminant of the pairing is $\langle b, b \rangle \in\bq$ where $b$ is a $\bz$-basis of $\mydet_\bz P$. Since
$$ \langle -b, -b \rangle=(-1)^2\langle b, b \rangle =\langle b, b \rangle $$
 the discriminant is a well-defined rational number (of absolute value $\chi^\times\Cone(\phi)$).
\label{rmk:disc} \end{rem} 

\begin{proof} Poincar\'e duality for algebraic de Rham cohomology of $\X_\bq/\bq$ is discussed in \cite{stacks}[Prop. 50.20.4]. It turns out that one can lift the construction of the cup product pairing in loc. cit. to the derived de Rham complex on $\X$ since we are truncating by $F^d$. More precisely, choose a simplicial resolution $P_\bullet\to\co_\X$ in $\X_{Zar}$ where $P_i$ is a free $\bz$-algebra in $\X_{Zar}$ and denote by
 $\Omega^{[n,m]}_{P_\bullet/\bz}$ the complex (of simplicial modules) $$\Omega^n_{P_\bullet/\bz}\to\cdots\to \Omega^m_{P_\bullet/\bz}$$ in degrees $[n,m]$, zero for $n>m$, where the differential is the de Rham differential. Define a complex of sheaves of abelian groups on $\X_{Zar}$
$$ L\Omega^{[n,m]}_{\X/\bz} :=\Tot^*_\bullet \Omega^{[n,m],\sim}_{P_\bullet/\bz}$$
so that
$$ L\wedge^nL_{\X/\bz}[-n]=L\Omega^{[n,n]}_{\X/\bz};\quad   R\Gamma_{dR}(\mathcal{X}/\mathbb{Z})/F^n = R\Gamma(\X,L\Omega^{[0,n-1]}_{\X/\bz}).$$
Here and in the following we denote by $M_\bullet^\sim$ the ($n$-tuple) chain complex associated to a ($n$-tuple) simplicial module $M_\bullet$ \cite{Illusie71}[1.1] and we decorate the (partial) totalization of an $n$-tuple complex with the indices that are contracted into one. We use the convention that totalization of an upper and a lower index leads to an upper index. As in \cite{stacks}[50.4.0.1] the wedge product on differential forms
induces a map of bicomplexes
\begin{align*} \Tot^{*,*}\Tot_{\bullet,\bullet}\Omega^{\star,\sim}_{P_\bullet/\bz} \otimes_\bz \Omega^{\star,\sim}_{P_\bullet/\bz}
= &\Tot^{*,*}\Tot_{\bullet,\bullet}\left(\Omega^{\star}_{P_\bullet/\bz} \otimes_\bz \Omega^{\star}_{P_\bullet/\bz}\right)^\sim\\
\xrightarrow{\sigma} &\Tot^{*,*}\left(\Delta \left(\Omega^{\star}_{P_\bullet/\bz} \otimes_\bz \Omega^{\star}_{P_\bullet/\bz}\right)\right)^\sim\\
=&\left(\Tot^{*,*}\Delta \left(\Omega^{\star}_{P_\bullet/\bz} \otimes_\bz \Omega^{\star}_{P_\bullet/\bz}\right)\right)^\sim\\
\xrightarrow{\cup} & \Omega^{\star,\sim}_{P_\bullet/\bz}\to \Omega^{[0,d-1],\sim}_{P_\bullet/\bz}
\end{align*} 
where $\sigma$ is induced by shuffle map $\Tot_{\bullet,\bullet}\left(M_\bullet\otimes N_\bullet\right)^\sim\to \left(\Delta\left(M_\bullet\otimes N_\bullet\right)\right)^\sim$ of \cite{Illusie71}[(1.2.2.1)] and $\Delta$ denotes the diagonal simplicial object of a bisimplicial object. Since we have truncated to degrees $\leq d-1$ the above pairing factors through a pairing
$$\Tot^{*,*}\Tot_{\bullet,\bullet} \Omega^{[n,d-1],\sim}_{P_\bullet/\bz} \otimes_\bz \Omega^{[0,d-1-n],\sim}_{P_\bullet/\bz}\to \Omega^{[0,d-1],\sim}_{P_\bullet/\bz}$$
and hence we obtain a pairing 
\begin{align*} L\Omega^{[n,d-1]}_{\X/\bz} \otimes^L_\bz L\Omega^{[0,d-1-n]}_{\X/\bz}=& \Tot^{*,*} \Tot^*_\bullet \Omega^{[n,d-1],\sim}_{P_\bullet/\bz} \otimes_\bz \Tot^*_\bullet \Omega^{[0,d-1-n],\sim}_{P_\bullet/\bz}\\
\simeq &\Tot^{*,*} \Tot^*_\bullet \Tot^*_\bullet\ \Omega^{[n,d-1],\sim}_{P_\bullet/\bz} \otimes_\bz  \Omega^{[0,d-1-n],\sim}_{P_\bullet/\bz}\\
\simeq &\Tot^{*,*}_{\bullet,\bullet}\  \Omega^{[n,d-1],\sim}_{P_\bullet/\bz} \otimes_\bz  \Omega^{[0,d-1-n],\sim}_{P_\bullet/\bz}\\
\simeq &\Tot^*_\bullet \Tot^{*,*}\Tot_{\bullet,\bullet}\  \Omega^{[n,d-1],\sim}_{P_\bullet/\bz} \otimes_\bz \Omega^{[0,d-1-n],\sim}_{P_\bullet/\bz}\\
\to &\Tot^*_\bullet\Omega^{[0,d-1],\sim}_{P_\bullet/\bz}=L\Omega^{[0,d-1]}_{\X/\bz}\end{align*} 
and an induced pairing on cohomology 
\begin{equation}\begin{CD} R\Gamma(\X,L\Omega^{[n,d-1]}_{\X/\bz}) \otimes^L_\bz R\Gamma(\X,L\Omega^{[0,d-1-n]}_{\X/\bz}) @>>> R\Gamma(\X,L\Omega^{[0,d-1]}_{X/\bz}) \\
\Vert @. \Vert @.\\
F^n/F^d \otimes_\bz^L R\Gamma_{dR}(\mathcal{X}/\mathbb{Z})/F^{d-n} @>>> R\Gamma_{dR}(\mathcal{X}/\mathbb{Z})/F^d.
\end{CD} \notag\end{equation}

\begin{lem} One has $H^{i}(\X,L\Omega^{[0,d-1]}_{\X/\bz})=0$ for $i>2d-2$.  Moreover, the natural map
$$H^{2d-2}(\X,L\wedge^{d-1} L_{\X/\bz}[-d+1]) \to  H^{2d-2}(\X,L\Omega^{[0,d-1]}_{\X/\bz}) $$
induces an isomorphism
$$g: H^{2d-2}(\X,L\wedge^{d-1} L_{\X/\bz}[-d+1])/\mathrm{tor} \cong   H^{2d-2}(\X,L\Omega^{[0,d-1]}_{\X/\bz})/\mathrm{tor}$$
and therefore a trace map
\begin{multline} R\Gamma(\X,L\Omega^{[0,d-1]}_{X/\bz})[2d-2]\to H^{2d-2}(\X,L\Omega^{[0,d-1]}_{\X/\bz})/\mathrm{tor}\xrightarrow{g^{-1}} \\H^{2d-2}(\X,L\wedge^{d-1} L_{\X/\bz}[-d+1])/\mathrm{tor} \xrightarrow{(\ref{can})_*} H^{2d-2}(\X,\omega_{\X/\bz}[-d+1])/\mathrm{tor}\xrightarrow{\trace}\bz\label{ddrtrace}\end{multline}
\end{lem}

\begin{proof} We first remark that $H^i(\X,\F)=0$ for $i\geq d$ and any coherent sheaf $\F$ on $\X$. Indeed, this is clear for $i>d$ since the cohomological dimension of $\X_{Zar}$ is $d$. Duality for $f:\X\to\Spec(\bz)$
\begin{equation} R\Hom_\bz(Rf_*\F,\bz)\cong R\Hom_{\X}(\F,\omega_{\X/\bz}[d-1]) \notag\end{equation}
evaluated in degree $-d$ 
$$  \Hom_\bz(H^d(\X,\F),\bz)\cong  H^{-1}R\Hom_{\X}(\F,\omega_{\X/\bz})=0$$
shows that $H^d(\X,\F)$ is torsion. Evaluation in degree $-d+1$ 
$$ 0\to \Ext^1(H^d(\X,\F),\bz)\to \Hom_\X(\F,\omega_{\X/\bz}) \to \Hom_\bz(H^{d-1}(\X,\F),\bz)\to 0$$
shows that $H^d(\X,\F)=0$ since $\omega_{\X/\bz}$ is a line bundle, $f$ is flat, and therefore $\Hom_\X(\F,\omega_{\X/\bz})$ is torsion free.

Since $L\wedge ^r L_{\X/\bz}$ is an object of the derived category of coherent sheaves concentrated in degrees $\leq 0$ we also have $H^i(\X,L\wedge ^rL_{\X/\bz})=0$ for $i\geq d$.
The exact triangle 
$$ R\Gamma(\X,L\wedge ^{r}L_{\X/\bz})[-r])\to R\Gamma(\X,L\Omega^{[n,r]}_{\X/\bz})\to R\Gamma(\X,L\Omega^{[n,r-1]}_{\X/\bz})\to$$
and an easy induction then show that 
$ H^i(\X,L\Omega^{[n,m]}_{\X/\bz})=0$ for $i\geq d+m$. In particular, the map
$$H^{2d-2}(\X,L\wedge^{d-1} L_{\X/\bz}[-d+1]) \to  H^{2d-2}(\X,L\Omega^{[0,d-1]}_{\X/\bz}) $$
is surjective and an isomorphism after tensoring with $\bq$ (see the proof of \cite{stacks}[Prop. 50.20.4]), hence induces an isomorphism
$$g: H^{2d-2}(\X,L\wedge^{d-1} L_{\X/\bz}[-d+1])/\mathrm{tor} \cong   H^{2d-2}(\X,L\Omega^{[0,d-1]}_{\X/\bz})/\mathrm{tor}. $$
\end{proof}

We now prove (\ref{discrepancy}) by downward induction on $n$  starting with the trivial case $n=d$. The induction step is provided by the diagram with exact rows and columns
\[\begin{CD} F^{n+1}/F^d @>\epsilon_{n+1}>> (F^0/F^{d-n-1})^*[-2d+2] @>>> \Cone(\epsilon_{n+1})\\
@VVV @VVV @VVV\\
F^n/F^d @>\epsilon_n >> (F^0/F^{d-n})^*[-2d+2] @>>> \Cone(\epsilon_{n})\\
@VVV @VVV @VVV\\
F^n/F^{n+1} @>>> (F^{d-n-1}/F^{d-n})^*[-2d+2] @>>> R\Gamma(\X,C_{\X/\bz}^n)[-n]
\end{CD}\]  
where the bottom exact triangle is $R\Gamma(\X,-)[-n]$ applied to (\ref{duality}) in view of (\ref{graded}) and coherent sheaf duality for $f:\X\to\Spec(\bz)$:
\begin{align*}(F^{d-n-1}/F^{d-n})^*[-2d+2]=&R\Hom(Rf_*(L\wedge^{d-1-n}L_{\X/\bz})[-d+1+n],\bz)[-2d+2]\\
\cong &R\Hom_\X(L\wedge^{d-1-n}L_{\X/\bz}[-d+1+n],\omega_{\X/\bz})[-d+1]\\
\cong &R\Gamma(\X,\underline{R\Hom}_\X(L\wedge^{d-1-n}L_{\X/\bz},\omega_{\X/\bz}))[-n].\end{align*}
By Theorem \ref{lichtenbaum} we have 
\[ \chi^\times ( \Cone(\epsilon_{n}))=\chi^\times ( \Cone(\epsilon_{n+1}))\cdot A(\X)\]
which gives $\chi^\times ( \Cone(\epsilon_{n}))=A(\X)^{d-n}$ by induction.
\end{proof}

For any $n\in \bz$ we have an exact triangle on the generic fibre $X=\X_\bq$
\begin{equation} F^n\to R\Gamma_{dR}(X/\mathbb{Q})\to R\Gamma_{dR}(X/\mathbb{Q})/F^n\label{genericfil}\end{equation}
and we also have a duality isomorphism (\ref{pd})$_\bq$ for any $n\in\bz$ since $F^n=F^d=0$ on the generic fibre for $n\geq d$.

\begin{cor} Let $n\in\bz$ and denote by $\lambda_{dR}$ the isomorphism
\begin{align*}&\left(\mathrm{det}^{-1}_{\mathbb{Z}} R\Gamma_{dR}(\mathcal{X}/\mathbb{Z})/F^n
\otimes\mathrm{det}_{\mathbb{Z}}R\Gamma_{dR}(\mathcal{X}/\mathbb{Z})/F^{d-n}\right)_\bq\\
\simeq\,\, &\mathrm{det}^{-1}_{\mathbb{Q}} R\Gamma_{dR}(X/\mathbb{Q})/F^n
\otimes\mathrm{det}_{\mathbb{Q}}R\Gamma_{dR}(X/\mathbb{Q})/F^{d-n}\\
\stackrel{(\ref{pd})_\bq}{\simeq}\,\, &\mathrm{det}^{-1}_{\mathbb{Q}} R\Gamma_{dR}(X/\mathbb{Q})/F^n
\otimes\mathrm{det}_{\mathbb{Q}}^{-1}F^n\\
\stackrel{(\ref{genericfil})}{\simeq}\,\, &\mathrm{det}^{-1}_{\mathbb{Q}} R\Gamma_{dR}(X/\mathbb{Q})\\
\simeq\,\,  & \left(\mathrm{det}^{-1}_{\mathbb{Z}} R\Gamma_{dR}(\mathcal{X}/\mathbb{Z})/F^d\right)_\bq.
\end{align*}
Then
\begin{multline*} \lambda_{dR}\left(\mathrm{det}^{-1}_{\mathbb{Z}} R\Gamma_{dR}(\mathcal{X}/\mathbb{Z})/F^n
\otimes\mathrm{det}_{\mathbb{Z}}R\Gamma_{dR}(\mathcal{X}/\mathbb{Z})/F^{d-n}\right)
= A(\X)^{d-n}\cdot\mathrm{det}^{-1}_{\mathbb{Z}} R\Gamma_{dR}(\mathcal{X}/\mathbb{Z})/F^d.\end{multline*}
\label{dRfactor}\end{cor}

\begin{proof} For $n\leq d$ this is clear from Prop. \ref{drduality} and the fact that (\ref{genericfil}) is the scalar extension to $\bq$ of the exact triangle
\[ F^n/F^d\to R\Gamma_{dR}(\mathcal{X}/\mathbb{Z})/F^d\to R\Gamma_{dR}(\mathcal{X}/\mathbb{Z})/F^n.\]
For $n>d$ we have $R\Gamma_{dR}(\X/\mathbb{Z})/F^{d-n}=0$ and an exact triangle
\[ F^d/F^n\to R\Gamma_{dR}(\mathcal{X}/\mathbb{Z})/F^n\to R\Gamma_{dR}(\mathcal{X}/\mathbb{Z})/F^d\]
where 
\begin{align*} \chi^\times(F^d/F^n)=&\prod_{r=d}^{n-1} \chi^\times \left(R\Gamma(\X,L\wedge ^r L_{\X/\bz}[-r])\right)\\
=&\prod_{r=d}^{n-1} \chi^\times \left(R\Gamma(\X,C^r_{\X/\bz}[-r-1])\right)\\
= &A(\X)^{d-n}\end{align*}
by (\ref{graded}) and Theorem \ref{lichtenbaum}. Hence
$$\mathrm{det}^{-1}_{\mathbb{Z}} R\Gamma_{dR}(\mathcal{X}/\mathbb{Z})/F^n = 
 A(\X)^{d-n}\cdot\mathrm{det}^{-1}_{\mathbb{Z}} R\Gamma_{dR}(\mathcal{X}/\mathbb{Z})/F^d$$
inside $\mathrm{det}^{-1}_{\mathbb{Q}} R\Gamma_{dR}(X/\mathbb{Q})$.

\end{proof}

\section{The archimedean Euler factor}\label{archimedean}

Following \cite{pr95}, for any pure $\br$-Hodge structure $M$ over $\br$ of weight $w(M)$ we define
\begin{align*} h_j(M)&=\dim F^j/F^{j+1}=h^{j,w(M)-j}(M)\\
d_\pm(M)&=\dim_\br M^{F_\infty=\pm 1}\\
t_H(M)&=\sum_jjh_j(M)=\frac{w(M)\cdot \dim_\br M}{2}=\frac{w(\det(M))}{2}\\
L_\infty(M,s)=&\prod_{p<q:=w(M)-p}\Gamma_{\bc}(s-p)^{h^{p,q}}\cdot\prod_{p=\frac{w(M)}{2}}
\Gamma_{\br}(s-p)^{h^{p,+}}\Gamma_{\br}(s-p+1)^{h^{p,-}}
\end{align*}
where
$$ \Gamma_\br(s)=\pi^{-s/2}\Gamma(s/2);\quad  \Gamma_{\bc}(s)=2(2\pi)^{-s}\Gamma(s). $$
Note that the factorization of $L_\infty(M,s)$ corresponds to the decomposition of $M$ into simple $\br$-Hodge structures over $\br$. 
Also recall the leading coefficient of the $\Gamma$-function at $j\in\bz$
\begin{equation} \Gamma^*(j)=\begin{cases} (j-1)! & j\geq 1\\ (-1)^j/(-j)! & j\leq 0\end{cases}  \label{gamma-leading}\end{equation}

\begin{lem}(see also \cite{pr95}[4.3.2, Lemme C.3.7]) For any pure $\br$-Hodge structure $M$ over $\br$ one has 
\[ \frac{L^*_\infty(M,0)}{L^*_\infty(M^*(1),0)}=\pm 2^{d_+(M)-d_-(M)}(2\pi)^{d_-(M)+t_H(M)}\prod_j\Gamma^*(-j)^{h_j(M)}\]
\label{pr}\end{lem}

\begin{proof} The functional equation of the $\Gamma$-function
\begin{equation}  \Gamma(s)\Gamma(1-s)=\frac{\pi}{\sin(\pi s)} \label{gammafe}\end{equation}
implies
\begin{equation} \Gamma_\bc(s)\Gamma_\bc(1-s)=\frac{2}{\sin(\pi s)};\quad\quad \Gamma_\br(1+s)\Gamma_\br(1-s)=\cos(\frac{\pi s}{2})^{-1} .         \notag\end{equation}
Hence
\begin{equation}\frac{\Gamma_\bc(s-p)}{\Gamma_\bc(-s-(-q-1))}= \frac{\Gamma_\bc(s-p)}{\Gamma_\bc(1-(s-q))}=\Gamma_\bc(s-p)\Gamma_\bc(s-q)\frac{\sin(\pi(s-q))}{2}.
\label{lhodgefe}\end{equation}
Using in addition the identity $\Gamma_\br(s)\Gamma_\br(s+1)=\Gamma_\bc(s)$ we find
\begin{equation}\frac{\Gamma_\br(s-p)}{\Gamma_\br(-s-(-p-1))}= \frac{\Gamma_\br(s-p)\Gamma_\br(s-p+1)}{\Gamma_\br(1-(s-p))\Gamma_\br(s-p+1)}=\Gamma_\bc(s-p)\cos(\frac{\pi(s-p)}{2})
\label{lhodgefe2}\end{equation}
and similarly
\begin{equation}\frac{\Gamma_\br(s-p+1)}{\Gamma_\br(-s-(-p-1)+1)}= \frac{\Gamma_\br(s-p+1)\Gamma_\br(s-p)}{\Gamma_\br(2-(s-p))\Gamma_\br(s-p)}=\Gamma_\bc(s-p)\cos(\frac{\pi(s-p-1)}{2}).
\label{lhodgefe3}\end{equation}

Every pure $\br$-Hodge structure $M$ over $\br$ is the direct sum of simple $\br$-Hodge structures. The simple $\br$-Hodge structures are  $M_{p,q}$ of dimension $2$ for integers $p<q$ and $M_{p,\pm}$ of dimension $1$ for integers $p$ (where $F_\infty$ operates via $\pm (-1)^p$). From the above definition of $L_\infty(M,s)$  and (\ref{lhodgefe}), (\ref{lhodgefe2}), (\ref{lhodgefe3}) we obtain the following table
\[
\begin{array}{c|c|c|c} M & M^*(1)  & L_\infty(M,s) & \frac{L_\infty(M,s)}{L_\infty(M^*(1),-s)} \\
\hline
M_{p,q} & M_{-p-1,-q-1} & \Gamma_\bc(s-p) & \Gamma_\bc(s-p)\Gamma_\bc(s-q)\cdot\frac{\sin(\pi(s-q))}{2} \\
\hline
M_{p,+} & M_{-p-1,+} & \Gamma_\br(s-p)  & \Gamma_\bc(s-p)\cdot\cos(\frac{\pi(s-p)}{2})  \\
\hline
M_{p,-} & M_{-p-1,-} & \Gamma_\br(s-p+1)  & \Gamma_\bc(s-p)\cdot\cos(\frac{\pi(s-p-1)}{2})  \\
\end{array}\]
We have 
\[ \left. \sin(\pi(s-q))\right|^*_{s=0}=(-1)^q\pi\]
and
\[ \left.\cos(\frac{\pi(s-p)}{2})\right|^*_{s=0}=\begin{cases}(-1)^{p/2} & \text{$p$ even}\\ (-1)^{(p-1)/2}\frac{\pi}{2} & \text{$p$ odd}\end{cases}\]

It is now straightforward to verify the entries of the following table which confirm Lemma \ref{pr} for simple $\br$-Hodge structures. Since all quantities are additive in $M$ the general case follows by writing $M$ as a sum of simple $\br$-Hodge structures.
\[
\begin{array}{c|c|c|c|c|c} M & d_+(M) &d_-(M) & h_j(M)& t_H(M)    &\frac{L^*_\infty(M,0)}{L^*_\infty(M^*(1),0)}\\
\hline
M_{p,q} & 1& 1 & \text{$1$ for $j=p,q$}&p+q  &  \pm(2\pi)^{p+q+1}\Gamma^*(-p)\Gamma^*(-q)\\
& & & \text{$0$ else}& & \\
\hline
M_{p,+} & & & &   &  \\
\text{$p$ even} &1 &0 &\text{$1$ for $j=p$}& p& \pm 2(2\pi)^p\Gamma^*(-p)\\
\text{$p$ odd} & 0& 1&\text{$1$ for $j=p$}& p& \pm 2(2\pi)^p\Gamma^*(-p)\cdot\frac{\pi}{2}\\
\hline
M_{p,-} & & & &   &  \\
\text{$p$ even} &0 &1 &\text{$1$ for $j=p$}& p& \pm 2(2\pi)^p\Gamma^*(-p)\cdot\frac{\pi}{2}\\
\text{$p$ odd} & 1& 0&\text{$1$ for $j=p$}& p& \pm 2(2\pi)^p\Gamma^*(-p)
\end{array}\]
\end{proof}

Suppose now $\X$ is a regular scheme, proper and flat over $\Spec(\bz)$ with generic fibre $X:=\X_\bq$. The archimedean Euler factor of $\X$ is defined as
\begin{equation} \zeta(\mathcal{X}_\infty,s)=\prod_{i\in\bz} L_\infty(h^i(X),s)^{(-1)^i}\label{zetaxinftydef}\end{equation}
where $h^i(X)$ is the $\br$-Hodge structure on $H^i(\X(\bc),\br)$.

\begin{cor} One has  
\[ \frac{\zeta^*(\mathcal{X}_\infty,n)}{\zeta^*(\mathcal{X}_\infty,d-n)}=\pm 2^{d_+(\X,n)-d_-(\X,n)}(2\pi)^{d_-(\X,n)+t_H(\X,n)}\prod_{p,q}\Gamma^*(n-p)^{h^{p,q}\cdot (-1)^{p+q}}\]
where
\[ d_\pm(\X,n)=\sum_i (-1)^i d_\pm(h^i(X)(n)),\quad t_H(\X,n)=\sum_i (-1)^i t_H(h^i(X)(n)) \]
and $h^i(X)(n)$ denotes the $\br$-Hodge structure on $H^i(\X(\bc),(2\pi i)^n\br)$.
\label{prcor}\end{cor}

\begin{proof} For $M=h^i(X)(n)$ one has $M^*(1)\cong h^{2d-2-i}(X)(d-n)$ and 
$$h_j(M)=h^{j,i-2n-j}(M)=h^{p-n,i-p-n}(M)=h^{p,i-p}=h^{p,q}$$
with $p+q=i$, $j=p-n$. Therefore Lemma \ref{pr} implies
\begin{align*}\frac{\zeta^*(\mathcal{X}_\infty,n)}{\zeta^*(\mathcal{X}_\infty,d-n)}
=&\prod_i\frac{L^*_\infty(h^i(X)(n),0)^{(-1)^i}}{L^*_\infty(h^{2d-2-i}(X)(d-n),0)^{(-1)^{2d-2-i}}}\\ 
=&2^{d_+(\X,n)-d_-(\X,n)}(2\pi)^{d_-(\X,n)+t_H(\X,n)}\prod_{p,q}\Gamma^*(n-p)^{h^{p,q}\cdot (-1)^{p+q}}.
\end{align*}
\end{proof}

\begin{lem} One has 
\[\frac{C(\mathcal{X},n)}{C(\mathcal{X},d-n)}=\pm\left(\prod_{p,q}\Gamma^*(n-p)^{h^{p,q}\cdot (-1)^{p+q}}\right)^{-1}\]
\label{clemma}\end{lem}

\begin{proof} Since $\X(\bc)$ is smooth proper of dimension $d-1$ the Hodge numbers $h^{p,q}$ are nonzero only for $0\leq p\leq d-1$. By definition (\ref{explicitcorrectingfactor})
\begin{align} C(\mathcal{X},n)^{-1}=&\prod_{0\leq p\leq n-1,q}(n-p-1)!^{h^{p,q}\cdot (-1)^{p+q}}\notag\\
=&\prod_{0\leq p\leq n-1,q}\Gamma^*(n-p)^{h^{p,q}\cdot (-1)^{p+q}}.\label{lowerhalf}\end{align}
On the other hand (\ref{gammafe}) implies 
$$ \Gamma^*(j)\Gamma^*(1-j)=\pm1$$
and therefore we have
\begin{align} \prod_{n\leq p\leq d-1,q}\Gamma^*(n-p)^{h^{p,q}\cdot (-1)^{p+q}}=&\pm\prod_{n\leq p\leq d-1,q}\Gamma^*(1-(n-p))^{-h^{p,q}\cdot (-1)^{p+q}}\notag\\
=&\pm\prod_{0\leq p\leq d-n-1,q}\Gamma^*(d-n-p)^{-h^{p,q}\cdot (-1)^{p+q}}\notag\\
=&\pm C(\X,d-n).
\label{upperhalf}\end{align}
Combining (\ref{lowerhalf}) and (\ref{upperhalf}) gives the Lemma.
\end{proof}

\section{The main result}\label{sec:main}

Recall the definition of the completed Zeta-function of $\X$
$$\zeta(\overline{\mathcal{X}},s):=\zeta(\mathcal{X},s)\cdot \zeta(\mathcal{X}_\infty,s)$$
where $\zeta(\mathcal{X}_\infty,s)$ was defined in (\ref{zetaxinftydef}).  We repeat Conjecture \ref{conj-funct-equation-intro} from the introduction.

\setcounter{section}{1}
\setcounter{thm}{2}
\begin{conj}\label{conj-funct-equation} (Functional Equation) Let $\X$ be a regular scheme of dimension $d$, proper and flat over $\Spec(\bz)$. Then $\zeta(\X,s)$ has a meromorphic continuation to all $s\in\bc$ and 
$$A(\mathcal{X})^{(d-s)/2}\cdot \zeta(\overline{\mathcal{X}},d-s)=\pm A(\mathcal{X})^{s/2} \cdot \zeta(\overline{\mathcal{X}},s).$$
\end{conj}

This conjecture is true for $d=1$ where it reduces to the functional equation of the Dedekind Zeta function. It is true for $d=2$ by \cite{Bloch87}[Prop. 1.1] provided that the L-function $L(h^1(\X_\bq),s)$ satisfies the expected functional equation involving the Artin conductor of the $l$-adic representation $H^1(\X_{\bar{\bq}},\bq_l)$. This is the case if $\X$ is a regular model of a potentially modular elliptic curve over a number field $F$ in view of the compatibility of the (local) Langlands correspondence for $GL_2$ with $\epsilon$-factors and hence conductors. Potential modularity of elliptic curves is known if $F$ is totally real or quadratic over a totally real field. We refer to  \cite{boxeretal}[1.1] for a discussion of these results and for the original references. In \cite{boxeretal} potential modularity is also shown for abelian surfaces over totally real fields $F$ and hence Conjecture \ref{conj-funct-equation} should hold for regular models of genus 2 curves over totally real fields $F$ (since this involves the local Langlands correspondence for $GSp_4/F$ we are unsure whether the conductor in the functional equation is indeed the Artin conductor).

We repeat Theorem \ref{thm:fe} from the introduction which is the main result of this paper.

\begin{thm}
Assume that $\zeta(\overline{\X},s)$ satisfies Conjecture \ref{conj-funct-equation}.
Then Conjecture \ref{conjmain} for $(\X,n)$ is equivalent to Conjecture \ref{conjmain} for $(\X,d-n)$.
\end{thm}

\begin{proof} The reduction of this theorem to Theorem \ref{thmintro} was already made in \cite{Flach-Morin-16}[5.7]. More precisely, recall the invertible $\bz$-module 
\begin{eqnarray*}
\Xi_{\infty}(\mathcal{X}/\mathbb{Z},n)&:=&\mathrm{det}_{\mathbb{Z}}R\Gamma_{W}(\mathcal{X}_{\infty},\mathbb{Z}(n))\otimes\mathrm{det}^{-1}_{\mathbb{Z}} R\Gamma_{dR}(\mathcal{X}/\mathbb{Z})/F^n \\
& &\otimes  \mathrm{det}^{-1}_{\mathbb{Z}}R\Gamma_{W}(\mathcal{X}_{\infty},\mathbb{Z}(d-n))
\otimes\mathrm{det}_{\mathbb{Z}}R\Gamma_{dR}(\mathcal{X}/\mathbb{Z})/F^{d-n},
\end{eqnarray*}
the canonical isomorphism
$$\Delta(\mathcal{X}/\mathbb{Z},n)\otimes \Xi_{\infty}(\mathcal{X}/\mathbb{Z},n)\stackrel{\sim}{\longrightarrow}\Delta(\mathcal{X}/\mathbb{Z},d-n)$$
and the canonical trivialization 
$$\xi_{\infty}:\mathbb{R}\stackrel{\sim}{\longrightarrow} \Xi_{\infty}(\mathcal{X}/\mathbb{Z},n)\otimes\mathbb{R}$$
which is compatible with the trivializations (\ref{introtriv}) of $\Delta(\X/\bz,n)$ and $\Delta(\X/\bz,d-n)$ (see  \cite{Flach-Morin-16}[Prop. 5.29]). 
Denoting by $$x_{\infty}(\mathcal{X},n)^2\in\mathbb{R}_{>0}$$
the strictly positive real number such that $$\xi_{\infty}(x_{\infty}(\mathcal{X},n)^{-2}\cdot \mathbb{Z})=\Xi_{\infty}(\mathcal{X}/\bz,n)$$
the canonical isomorphism
$$\Xi_{\infty}(\mathcal{X}/\mathbb{Z},n)\otimes\Xi_{\infty}(\mathcal{X}/\mathbb{Z},d-n)\cong\bz$$
implies that
\begin{equation}x_{\infty}(\mathcal{X},n)^2\cdot x_{\infty}(\mathcal{X},d-n)^2=1.\label{sym}\end{equation}
It was then shown in \cite{Flach-Morin-16}[Cor. 5.31] that Theorem \ref{thm:fe} is equivalent to the following identity (note that there is a typo in the statement of \cite{Flach-Morin-16}[Cor. 5.31] and $C(\X,n)$ and $C(\X,d-n)$ have to be replaced by their inverses)
\begin{equation}\notag
\begin{split}
 & A(\mathcal{X})^{n/2} \cdot \zeta^*(\mathcal{X}_\infty,n)\cdot x_{\infty}(\mathcal{X},n)^{-1}\cdot C(\mathcal{X},n)
 \\
  =&\pm   A(\mathcal{X})^{(d-n)/2}\cdot \zeta^*(\mathcal{X}_\infty,d-n)\cdot x_{\infty}(\mathcal{X},d-n)^{-1}\cdot C(\mathcal{X},d-n).
 \end{split}
\end{equation}
But using (\ref{sym}) this identity is equivalent to the identity of Thm. \ref{thmintro}. This already concludes the proof of Thm. \ref{thm:fe}.
\end{proof}

It remains to prove Theorem \ref{thmintro} which we repeat here for the convenience of the reader.
\setcounter{thm}{1}

\begin{thm}
Let $\X$ be a regular scheme of dimension $d$, proper and flat over $\Spec(\bz)$.  Then we have 
\begin{equation*}
x_{\infty}(\mathcal{X},n)^{2}=\pm A(\mathcal{X})^{n-d/2} \cdot \frac{\zeta^*(\mathcal{X}_\infty,n)}{\zeta^*(\mathcal{X}_\infty,d-n)}\cdot \frac{C(\mathcal{X},n)}{C(\mathcal{X},d-n)}.
\end{equation*}
\end{thm}

\setcounter{section}{5}
\setcounter{thm}{0}
\begin{proof}
 By Corollary \ref{prcor} and Lemma \ref{clemma}  this identity is equivalent to 
\begin{equation}\label{toward3}
x_{\infty}(\mathcal{X},n)^{2}=\pm A(\mathcal{X})^{n-d/2} \cdot 2^{d_+(\X,n)-d_-(\X,n)}\cdot(2\pi)^{d_-(\X,n)+t_H(\X,n)}.
\end{equation}

\begin{lem} The isomorphism $\xi_\infty$ is induced by the sequence of isomorphisms
\begin{align} &\left(\mydet_\bz R\Gamma_W(\mathcal{X}_{\infty},\mathbb{Z}(n))\otimes  \mathrm{det}^{-1}_{\mathbb{Z}}R\Gamma_{W}(\mathcal{X}_{\infty},\mathbb{Z}(d-n))\right)_\br\label{deRhamstructure}\\
 \xrightarrow{(\ref{firstiso})_\br} \,\,&\mydet_\br \left( R\Gamma(\mathcal{X}(\mathbb{C}),\br(n))^+\oplus R\Gamma(\mathcal{X}(\mathbb{C}),\br(n-1))^+\right) \notag\\
 \xrightarrow{(\ref{pmdecomp})}\,\, &\mydet_\br R\Gamma(\mathcal{X}(\mathbb{C}),\bc)^+ \notag\\
 \xrightarrow{(\ref{bettidR})^+}\,\, &\mydet_\br R\Gamma_{dR}(\mathcal{X}_\bc/\bc)^+ \notag\\
 \simeq\,\, &\mydet_\br R\Gamma_{dR}(\mathcal{X}_\br/\br)\simeq\left(\mydet_{\mathbb{Z}} R\Gamma_{dR}(\mathcal{X}/\mathbb{Z})/F^d\right)_\br\notag\\
 \xrightarrow{\lambda_{dR}^{-1}}\,\, &\left(\mathrm{det}_{\mathbb{Z}} R\Gamma_{dR}(\mathcal{X}/\mathbb{Z})/F^n
\otimes\mathrm{det}_{\mathbb{Z}}^{-1}R\Gamma_{dR}(\mathcal{X}/\mathbb{Z})/F^{d-n}\right)_\br\notag
 \end{align}
 where $\lambda_{dR}$ was defined in Cor. \ref{dRfactor}.
\label{xilemma}\end{lem}

\begin{proof} The isomorphism $\xi_\infty$ was defined in \cite{Flach-Morin-16}[Prop. 5.29] 
\begin{eqnarray*}
&&\left(\mathrm{det}_{\mathbb{Z}}R\Gamma_{W}(\mathcal{X}_{\infty},\mathbb{Z}(n))\otimes_{\mathbb{Z}}\mathrm{det}^{-1}_{\mathbb{Z}} R\Gamma_{dR}(\mathcal{X}/\mathbb{Z})/F^n\right)_\mathbb{R}\\
&\simeq&
\mathrm{det}_{\mathbb{R}}R\Gamma_{\mathcal{D}}(\mathcal{X}/\mathbb{R},\mathbb{R}(n))\\
&\simeq&
\mathrm{det}_{\mathbb{R}} R\Gamma_{\mathcal{D}}(\mathcal{X}/\mathbb{R},\mathbb{R}(d-n))^*[-2d+1]\\
&\simeq&
\mathrm{det}_{\mathbb{R}}R\Gamma_{\mathcal{D}}(\mathcal{X}/\mathbb{R},\mathbb{R}(d-n))\\
&\simeq&\left(\mathrm{det}_{\mathbb{Z}}R\Gamma_{W}(\mathcal{X}_{\infty},\mathbb{Z}(d-n))\otimes_{\mathbb{Z}}\mathrm{det}_{\mathbb{Z}}^{-1} R\Gamma_{dR}(\mathcal{X}/\mathbb{Z})/F^{d-n}\right)_\mathbb{R}
\end{eqnarray*}
using the defining exact triangle
$$ (R\Gamma_{dR}(\mathcal{X}/\mathbb{Z})/F^{n})_\br[-1] \to R\Gamma_{\mathcal{D}}(\mathcal{X}_{/\mathbb{R}},\mathbb{R}(n))\to R\Gamma(\mathcal{X}(\mathbb{C}),\br(n))^+ $$ and duality 
$$R\Gamma_{\mathcal{D}}(\mathcal{X}_{/\mathbb{R}},\mathbb{R}(n))\simeq R\Gamma_{\mathcal{D}}(\mathcal{X}_{/\mathbb{R}},\mathbb{R}(d-n))^*[-2d+1]$$
for Deligne cohomology. This duality is constructed in \cite{Flach-Morin-16}[Lemma 2.3] by taking $G_\br$-invariants in 
\begin{equation} R\Gamma_{\mathcal{D}}(\mathcal{X}_{/\mathbb{C}},\mathbb{R}(n))\simeq R\Gamma_{\mathcal{D}}(\mathcal{X}_{/\mathbb{C}},\mathbb{R}(d-n))^*[-2d+1] \label{deldual}\end{equation}
which is obtained as follows. Dualizing the defining exact triangle
$$ (R\Gamma_{dR}(\mathcal{X}/\mathbb{Z})/F^{d-n})_\bc[-1] \to R\Gamma_{\mathcal{D}}(\mathcal{X}_{/\mathbb{C}},\mathbb{R}(d-n))\to R\Gamma(\mathcal{X}(\mathbb{C}),\br(d-n)) $$
and using Poincar\'e duality (\ref{bettipd}) and (\ref{pd})$_\bc$ on $\X(\bc)$ we obtain the bottom exact triangle in the diagram
$$ \begin{CD}  {} @. R\Gamma(\mathcal{X}(\mathbb{C}),\br(n))[-1]@=R\Gamma(\mathcal{X}(\mathbb{C}),\br(n))[-1] \\
@. @VVV @VVV \\
F^n_\bc @<<< (R\Gamma_{dR}(\mathcal{X}/\mathbb{Z})/F^{n})_\bc[-1]  @<\beta<<  R\Gamma(\mathcal{X}(\mathbb{C}),\bc)[-1]\\
\Vert @. @VVV @VVV\\
F^n_\bc @<<< R\Gamma_{\mathcal{D}}(\mathcal{X}_{/\mathbb{C}},\mathbb{R}(d-n))^*[-2d+1] @<<< R\Gamma(\mathcal{X}(\mathbb{C}),\br(n-1))[-1] \end{CD}$$
The right hand column is induced by the decomposition 
\begin{equation} \bc\cong \br(n)\oplus\br(n-1) \label{pmdecomp}\end{equation}
on coefficients, and the map $\beta$ is the comparison isomorphism
\begin{equation} R\Gamma(\mathcal{X}(\mathbb{C}),\bc) \simeq R\Gamma_{dR}(\mathcal{X}_\bc/\mathbb{C})\label{bettidR}\end{equation}
composed with the natural projection. It is then clear that all rows and columns in the diagram are exact, and the middle column is the defining exact triangle of 
$R\Gamma_{\mathcal{D}}(\mathcal{X}_{/\mathbb{C}},\mathbb{R}(n))$, giving (\ref{deldual}). Recalling that (\ref{firstiso}) was also defined using  Poincar\'e duality (\ref{bettipd}) we find that the isomorphisms used in (\ref{deRhamstructure}) are precisely those used in the construction of (\ref{deldual})$^+$.
\end{proof}

We call the real line $\mydet_\br R\Gamma(\mathcal{X}(\mathbb{C}),\bc)^+$ the de Rham real structure of $\mydet_\bc R\Gamma(\X(\bc),\bc)$
and the real line $\mydet_\br R\Gamma(\X(\bc),\br(n))$ the Betti real structure of $\mydet_\bc R\Gamma(\X(\bc),\bc)$. By (\ref{noncan}) we have 
\begin{equation}\mydet_\br R\Gamma(\mathcal{X}(\mathbb{C}),\bc)^+ \cdot (2\pi i)^{d_{-}(\X,n)}= \mydet_\br R\Gamma(\X(\bc),\br(n)).\label{BdR}\end{equation}
In the remaining computations of the proof of Theorem \ref{thmintro} all identities should be understood up to sign. We choose bases of the various $\bz$-structures of the de Rham real line appearing in (\ref{deRhamstructure}) 
\begin{align*} \bz\cdot\tilde{b}_B=\,\,&\mydet_\bz R\Gamma_W(\mathcal{X}_{\infty},\mathbb{Z}(n))\otimes  \mathrm{det}^{-1}_{\mathbb{Z}}R\Gamma_{W}(\mathcal{X}_{\infty},\mathbb{Z}(d-n))\notag\\
 \bz\cdot b_{dR}  =\,\,& \mydet_{\mathbb{Z}} R\Gamma_{dR}(\mathcal{X}/\mathbb{Z})/F^d\notag\\
 \bz\cdot\tilde{b}_{dR} =\,\,&\mathrm{det}_{\mathbb{Z}} R\Gamma_{dR}(\mathcal{X}/\mathbb{Z})/F^n
\otimes\mathrm{det}_{\mathbb{Z}}^{-1}R\Gamma_{dR}(\mathcal{X}/\mathbb{Z})/F^{d-n}\notag
 \end{align*}
and we also choose a basis
$$ \bz\cdot b_B=\mydet_\bz R\Gamma(\X(\bc),\bz(n))$$
of the natural $\bz$-structure in the Betti real structure. Let $P\in\bc^\times$ be the Betti-de Rham period, i.e. we have
$$ b_{dR}=P\cdot b_B$$
under the comparison isomorphism (\ref{bettidR}). 

\begin{lem} Let $\varepsilon_B\in\{\pm 1\}$ be the discriminant (see Remark \ref{rmk:disc}) of the Poincar\'e duality pairing 
$$ R\Gamma(\X(\bc),\bz(n))\otimes R\Gamma(\X(\bc),\bz(n)) \xrightarrow{\mathrm{Tr}\circ\cup}\bz[-2d+2] $$
and $\varepsilon_{dR}\cdot A(\X)^d$ the discriminant of the deRham duality pairing (\ref{disc}). Then
$$ P=\sqrt{\varepsilon_B\varepsilon_{dR}} \cdot (2\pi i)^{t_H(\X,n)} \cdot A(\X)^{\frac{d}{2}}.$$
Moreover $P\cdot (2\pi i)^{d_{-}(\X,n)}$ is real and hence we have
$$P\cdot (2\pi i)^{d_{-}(\X,n)}=(2\pi)^{d_{-}(\X,n)+t_H(\X,n)} \cdot A(\X)^{\frac{d}{2}}.$$
\end{lem} 

\begin{proof}  We have a commutative diagram 
$$ \begin{CD} R\Gamma(\X(\bc),\bz(n))\otimes R\Gamma(\X(\bc),\bz(n)) @>{\mathrm{Tr}\circ\cup}>>\bz(2n-d+1)[-2d+2] \\
@VVV  @VVV\\
R\Gamma(\X(\bc),\bc)\otimes R\Gamma(\X(\bc),\bc) @>{\mathrm{Tr}\circ\cup}>>\bc[-2d+2]\\
@AAA @AAA\\
R\Gamma_{dR}(\X/\bz)/F^d\otimes_\bz^L R\Gamma_{dR}(\X/\bz)/F^d @>(\ref{disc})>> \bz[-2d+2]
\end{CD}$$ 
where the bottom square commutes since the comparison isomorphism 
$$R\Gamma(\mathcal{X}(\mathbb{C}),\bc) \simeq R\Gamma_{dR}(\mathcal{X}_\bc/\mathbb{C})\simeq R\Gamma_{dR}(\mathcal{X}_\bq/\bq)_\bc$$
is compatible with cup product and cycle classes, and the trace map sends the cycle class of a closed point to its degree over the base field. We also use the fact that the trace map in algebraic de Rham cohomology 
$$H^{2d-2}_{dR}(\X_\bq/\bq)\xleftarrow{\sim}H^{d-1}(\X_\bq,\Omega_{\X_\bq/\bq}^{d-1})\xrightarrow{\mathrm{Tr}} \bq$$
is the base change of the Trace map (\ref{ddrtrace}) under $\Spec(\bq)\to\Spec(\bz)$ by \cite{verdier68}. Applying the construction of Remark \ref{rmk:disc} we then obtain a pairing 
$$ \langle \cdot,\cdot \rangle : \mydet_\bc R\Gamma(\X(\bc),\bc)\otimes_\bc\mydet_\bc R\Gamma(\X(\bc),\bc) \simeq \bc$$
which restricts to the corresponding $\bq$-valued pairing on $\mydet_\bz R\Gamma_{dR}(\X/\bz)/F^d$ and a $\bq\cdot(2\pi i)^{(2n-d+1)\chi}$-valued pairing on $\mydet_\bz R\Gamma(\X(\bc),\bz(n))$. Here 
$$\chi=\rank_\bz R\Gamma(\X(\bc),\bz(n)) = \dim_\br R\Gamma(\X(\bc),\br(n))= \sum_i (-1)^i \dim_\br H^i(\X(\bc),\br(n)).$$

 We then have
$$ \varepsilon_{dR}\cdot A(\X)^d=\langle b_{dR},b_{dR}\rangle = P^2 \langle b_B,b_B\rangle= P^2 \varepsilon_B (2\pi i)^{(2n-d+1)\chi}$$
and moreover
\begin{align*} -(2n-d+1)\chi =& \sum_{i<d-1} (-1)^i (i-2n + 2d-2-i-2n) \dim_\br H^i(\X(\bc),\br(n)) \\
+&(-1)^{d-1}(d-1-2n)\dim_\br H^{d-1}(\X(\bc),\br(n))\\
=&2\,t_H(\X,n).
\end{align*}
Hence
$$ P^2 = \varepsilon_{dR}\varepsilon_B \cdot (2\pi i)^{2\,t_H(\X,n)}\cdot A(\X)^d$$
which shows the first statement.
Since both $b_B$ and $b_{dR}\cdot (2\pi i)^{d_{-}(\X,n)} = P\cdot (2\pi i)^{d_{-}(\X,n)}\cdot b_B$ lie in the Betti real structure, the factor
$P\cdot (2\pi i)^{d_{-}(\X,n)}$ is real. This proves the second statement.
\end{proof} 

\begin{rem} From the Lemma we deduce
$$ \varepsilon_B\varepsilon_{dR}=(-1)^{d_{-}(\X,n)+t_H(\X,n)} =(-1)^{d_{-}(\X,0)+\frac{d-1}{2}\chi}.$$
This generalizes the classical fact that the sign of the discriminant of a number field $F$ is $(-1)^{r_2}$ where $r_2=d_-(\Spec(\co_F),0)$ is the number of complex places. In this case $\varepsilon_B=1$.
\end{rem}

We can now  finish the proof of Theorem \ref{thmintro} by verifying (\ref{toward3}).  By Prop. \ref{final-bettiprop} we have
$$ \tilde{b}_{B}\cdot (2\pi i)^{d_{-}(\X,n)} = b_B \cdot 2^{d_-(\X,n)-d_+(\X,n)}$$
and by Corollary \ref{dRfactor} 
$$ \tilde{b}^{-1}_{dR} = A(\X)^{d-n}\cdot b^{-1}_{dR}.$$
Therefore
\begin{align*} x_\infty(\X,n)^{-2} = & \tilde{b}_B\cdot \tilde{b}_{dR}^{-1} =  (2\pi i)^{-d_{-}(\X,n)} \cdot 2^{d_-(\X,n)-d_+(\X,n)}\cdot b_B\cdot A(\X)^{d-n}\cdot b^{-1}_{dR}\\
= &  (2\pi i)^{-d_{-}(\X,n)} \cdot 2^{d_-(\X,n)-d_+(\X,n)}\cdot  A(\X)^{d-n}\cdot P^{-1}\\
= & (2\pi i)^{-d_{-}(\X,n)} \cdot 2^{d_-(\X,n)-d_+(\X,n)}\cdot  A(\X)^{d-n}\cdot 
\sqrt{\varepsilon_B\varepsilon_{dR}} \cdot (2\pi i)^{-t_H(\X,n)}\cdot A(\X)^{-\frac{d}{2}}\\
=& A(\X)^{\frac{d}{2}-n}\cdot 2^{d_-(\X,n)-d_+(\X,n)}\cdot  (2\pi)^{-d_{-}(\X,n)-t_H(\X,n)} 
\end{align*}
which is (\ref{toward3}).
\end{proof}

\end{document}